%% file: GeometricOrlov.tex
\documentclass{scrartcl}

\input xy
\xyoption{all}

\usepackage[letterpaper,margin=1in]{geometry}
\usepackage{amsmath}
\usepackage[usenames,dvipsnames]{color}
\usepackage[colorlinks, linkcolor=Cerulean, citecolor=Violet, urlcolor=LimeGreen]{hyperref}
\usepackage{comment}
\usepackage[OT1]{fontenc}
\usepackage{gfsartemisia-euler}
\usepackage{amsmath,amssymb}

\author{Ian Shipman}
\date{\today}

\def\mc{\mathcal}
\def\frk{\mathfrak}
\def\wt{\widetilde}
\def\wh{\widehat}
\def\mb{\mathbf}

\def\ol{\overline}

\input{standard_include.tex}

\usepackage{stmaryrd}

\title{A Geometric approach to Orlov's theorem}

\DeclareMathOperator{\DBr}{DBr^\Gamma}

\def\GIT{\sslash}

\def\OZgr{\O_{\Zh}^{gr}}
\def\S{\cK}

\DeclareMathOperator{\DQcoh}{\mb{D}Qcoh^\Gamma}
\DeclareMathOperator{\grDcoh}{\mb{D}coh^\Gamma}
\DeclareMathOperator{\grQcoh}{Qcoh^\Gamma}
\DeclareMathOperator{\MF}{MF^\Gamma}
\DeclareMathOperator{\DMF}{\mb{D}MF^\Gamma}

\begin{document}

\maketitle

\begin{abstract}
A famous theorem of D. Orlov describes the derived bounded category of coherent sheaves on projective hypersurfaces in terms of an algebraic construction called graded matrix factorizations.  In this article, I implement a proposal of E. Segal to prove Orlov's theorem in the Calabi-Yau setting using a globalization of the category of graded matrix factorizations (graded D-branes).  Let $X \subset \P$ be a projective hypersurface.  Already, Segal has established an equivalence between Orlov's category of graded matrix factorizations and the category of graded D-branes on the canonical bundle $K_{\P}$ to $\P$.  To complete the picture, I give an equivalence between the homotopy category of graded D-branes on $K_{\P}$ and $\Dcoh(X)$.  This can be achieved directly and by deforming $K_{\P}$ to the normal bundle of $X \subset K_{\P}$ and invoking a global version of Kn\"{o}rrer periodicity.  We also discuss an equivalence between graded D-branes on a general smooth quasi-projective variety and on the formal neighborhood of the singular locus of the zero fiber of the potential.
\end{abstract}

\tableofcontents

\section{Introduction}
\input{GO_introduction.tex}

\section{Construction of the category of graded D-branes}\label{s.const}
\input{GO_const.tex}

\section{Graded D-branes for a vector bundle and section}\label{s.vb}
\input{GO_vb.tex}


\section{Deformation to the normal bundle}\label{s.dnb}
\input{GO_def.tex}

\section{Application to Orlov's theorem}\label{s.app}
\input{GO_revisit.tex}

\section{Localization}\label{s.loc}
\input{GO_loc.tex}

\bibliographystyle{alpha}
\bibliography{branes.bib}

\end{document}

%% file: standard_include.tex
\def\onto{\twoheadrightarrow}
\def\into{\hookrightarrow}

\def\from{\leftarrow}

\def\tensor{\otimes}

\def\bt{\bullet}
\def\he{\simeq}
\def\del{\partial}

\DeclareMathOperator{\Hom}{Hom}

\DeclareMathOperator{\End}{End}

\DeclareMathOperator{\Der}{Der}

\def\sHom{\mc{H}om}
\def\sExt{\mc{E}xt}
\def\sEnd{\mc{E}nd}

\DeclareMathOperator{\grHom}{grHom}

\DeclareMathOperator{\Perf}{\frk{Perf}}
\DeclareMathOperator{\Dcoh}{\mb{D}^b coh}

\DeclareMathOperator{\coh}{coh}
\DeclareMathOperator{\Qcoh}{Qcoh}

\DeclareMathOperator{\spec}{Spec}

\DeclareMathOperator{\Sym}{Sym}
\def\invlim{\underleftarrow{\lim}}

\DeclareMathOperator{\id}{id}

\DeclareMathOperator{\coker}{coker}
\DeclareMathOperator{\im}{im}

\DeclareMathOperator{\cn}{cone}

\def\dim{\text{dim}}

\def\svC{\check{\mc{C}}}

\def\Cx{\C^\times}
\DeclareMathOperator{\GL}{GL}
\DeclareMathOperator{\SL}{SL}

\def\Z{\mathbb{Z}}

\def\C{\mathbb{C}}
\def\P{\mathbb{P}}
\def\A{\mathbb{A}}

\def\cC{\mc{C}}

\def\cE{\mc{E}}
\def\cF{\mc{F}}
\def\cG{\mc{G}}
\def\cH{\mc{H}}
\def\cI{\mc{I}}

\def\cK{\mc{K}}
\def\cL{\mc{L}}
\def\O{\mc{O}}
\def\cP{\mc{P}}
\def\cQ{\mc{Q}}

\def\cV{\mc{V}}
\def\cW{\mc{W}}

\def\bD{\mathbf{D}}

\def\Ft{\wt{F}}

\def\St{\wt{S}}

\def\Zh{\wh{Z}}

\def\Sbr{\overline{S}}

\def\sfH{{\sf H}}

\def\vC{\check{C}}

\usepackage{amsthm}

\theoremstyle{plain}
\newtheorem{thm}{Theorem}[section]
\newtheorem*{thm*}{Theorem}
\newtheorem{lma}[thm]{Lemma}
\newtheorem{prop}[thm]{Proposition}
\newtheorem{cor}[thm]{Corollary}

\theoremstyle{definition}
\newtheorem{dfn}[thm]{Definition}

\newtheorem{rmk}[thm]{Remark}
\newtheorem*{rmk*}{Remark}

\newtheorem{exmp}[thm]{Example}

%% file: GO_introduction.tex
The work in this paper was motivated by an attempt to understand Theorem 3.11 in \cite{Or} proposed by Ed Segal in \cite{Se}.  In order to state the theorem, we need to set up some elementary preliminaries.  Let $X \subset \P^n$ be a projective hypersurface of degree $w$ with defining equation $W \in R := \C[x_0,\dotsc,x_n]$.  The version of Orlov's theorem in which we are interested is a comparison between $\Dcoh(X)$ and the category of graded matrix factorization of $W$.

In order to define the category of graded matrix factorizations, it is convenient to first define the plain category of (ungraded) matrix factorizations (after \cite{E}), then describe the modifications necessary to yield the graded case.  A \emph{matrix factorization} of $W$ is a $\Z/2\Z$ graded free $R$ module $P$ together with an odd (or degree 1) $R$-linear endomorphism $d_P$ which satisfies $d_P^2 = W \cdot \id_P$.  If $P,Q$ are matrix factorizations then $\Hom_R(P,Q)$ is again a $\Z/2\Z$ graded free $R$ module.  Moreover there is a natural differential on $\Hom_R(P,Q)$ given on homogeneous morphisms $\phi$ by $d(\phi)=d_Q \circ \phi - (-1)^{\deg(\phi)} \phi \circ d_P$.  Thus, $\Hom_R(P,Q)$ is a $\Z/2\Z$ graded complex.  The \emph{category $MF(W)$ of matrix factorizations} is the $\C$-linear category whose objects are matrix factorizations of $W$ and where the vector space of morphisms between $P$ and $Q$ is $\sfH^0( \Hom_R(P,Q), d )$.  This category admits a natural triangulated structure in which $[2] \cong \id$.

Now, view $R$ as a graded ring with $\deg(x_i)=1$.  If $M$ is a graded $R$ module, we write $M(k)$ for the shifted module with $M(k)_i = M_{k+i}$.  A grading-preserving map $\phi:M \to M'$ induces a grading preserving map $M(k) \to M'(k)$, which we denote by $\phi(k)$.  Let $P$ be a matrix factorization of $W$ whose underlying module is graded.  Note that since the degree of $W$ is $w \neq 0$, $d$ cannot be graded.  However since $P$ is $\Z/2\Z$ graded as well, the map $d$ breaks into a pair of maps
\[ P_0 \stackrel{\alpha}{\to} P_1 \stackrel{\beta}{\to} P_0. \]
The matrix factorization $P$ is a \emph{graded matrix factorization} when $\alpha$ is graded and $\beta$ has degree $w$ in the sense that $\alpha:P_1 \to P_0(w)$ is graded.  The \emph{category $\MF(W)$ of graded matrix factorizations} (after \cite{Or, HoWa}) is the category whose objects are graded matrix factorizations and where the space of morphisms between $P$ and $Q$ is $\sfH^0(\Hom_R^{gr}(P,Q))$.  The superscript $^{gr}$ denotes the subspace of degree-preserving module maps.  The category of graded matrix factorizations also admits a triangulated structure.  However the ``homological'' shift interacts non-trivially with the $R$ module shift in the sense that $[2] \cong (w)$.  So, the category of graded matrix factorizations is fully $\Z$ graded.

\begin{thm*}[3.11 \cite{Or}]
Let $X \subset \P^n$ be a nonsingular hypersurface of degree $w$ with defining equation $W$.
\begin{enumerate}
 \item If $w < n+1$ then there is a fully faithful exact functor $\MF(W) \to \Dcoh(X)$.
 \item If $w = n+1$ then there is an equivalence $\Dcoh(X) \cong \MF(W)$.
 \item If $w > n+1$ then there is a fully faithful exact functor $\Dcoh(X) \to \MF(W)$.
\end{enumerate}
\end{thm*}

\begin{rmk}
This is an imprecise version of Orlov's theorem.  He not only constructs the fully faithful comparison functors but also describes the orthogonals to their essential images.
\end{rmk}

Inspired by the work of Witten \cite{W93}, Segal proposed an alternative proof of this theorem, focusing on the Calabi-Yau case.  There are two parts to his proposition and he completed the first part in his article \cite{Se}. He first provides a framework for studying graded matrix factorizations over schemes and stacks.  In this context, we adopt the terminology of the physics community and refer to these analogs of graded matrix factorizations as graded D-branes.  In this framework, a DG category $\DBr(Y,F)$ is associated to any scheme (or algebraic stack) $Y$ equipped with a regular function $F$ and some other data responsible for the graded structure.  We will discuss the construction in the next section and in this section we will ignore the grading data to simplify the presentation.  Let $\mu_w$ be the group of $w$-th roots of unity acting by scaling on $\A^{n+1}$.  The DG category associated to the quotient stack $[\A^{n+1}/\mu_w]$ and $W$ (which descends to the quotient stack since  $W$ is invariant under $\mu_w$) is a DG enhancement of the triangulated category of graded matrix factorizations.

Write $\P = \P^n$.  We can view $W$ as a section of $\O_\P(w)$ and thus as a regular function $W$ on $Y= \O_\P(-w)$ which is linear on the fibers of the projection $Y \to \P$.  There is a canonical birational morphism $Y \to \A^{n+1}\GIT\mu_w$ and the function $W$ on $Y$ is the pullback of the function induced by $W$ on $\A^{n+1}\GIT \mu_w$.  However there is a closer relationship between $Y$ and $[\A^{n+1}/\mu_w]$.  They are the GIT quotients of $\A^{n+1} \times \A^1$ under the action of $\C^\times$ with respect to the identity and inverse characters and their distinguished functions descend from a function on $\A^{n+1} \times \A^1$.  Segal formulates a principle that if $X_1 \dashrightarrow X_2$ is a rational, birational morphism between Calabi-Yau stacks identifying $W_1$ and $W_2$, then there should be a corresponding quasi-equivalence $\DBr(X_1,W_1) \he \DBr(X_2, W_2)$.  This motivates the following theorem.

\begin{thm*}[3.3 \cite{Se}]
Assume that $w = n+1$.  Then there is a family of quasi-equivalences $$\DBr([\A^{n+1} / \mu_w],W) \he \DBr(Y,W)$$ indexed by $\Z$.
\end{thm*}
\begin{rmk}
Segal's theorem holds for any action of $\Cx$ on an vector space where the sum of the weights is zero.  In fact his method establishes a trichotomy as in Orlov's theorem, depending on the sign of the sum of the weights.
\end{rmk}

The next step in the proposal is to construct an equivalence $[\DBr(Y,W)] \he \Dcoh(X)$.  We achieve this by a direct argument.  However we also explore a geometric technique in which the data used to define the categories mutates in a pleasant way and the functors used to compare the categories are extremely simple.  With respect to the embedding of $X$ into $Y$ along the zero section, the normal bundle $N_{X/Y}$ is isomorphic to $\O_X(w) \oplus \O_X(-w)$.  Let $p$ be a local coordinate on $Y=\O_\P(-w)$ which is linear along the fiber and which vanishes to first order along the zero section.  Then locally $W$ has the form $fp$ where $f$ is a function that is constant along the fibers and which vanishes to first order along $X$.  So $W$ vanishes to second order along $X$ and upon degeneration induces a homogeneous function $\bar{W}$ of weight 2 on the normal bundle.  In fact the induced function is a tautological function, the distinguished section of $\O_X(w) \tensor \O_X(-w) \subset \Sym^2 N^\vee_{X/Y}$ corresponding to the trivialization $\O_X(w)\tensor\O_X(-w) \cong \O_X$.

In order to establish a quasi-equivalence $\DBr(Y,W) \he \Perf(X)$ using geometry, we first establish a quasi-equivalence $\DBr(Y,W) \he \DBr(N_{X/Y},\bar{W})$.  A standard construction, the deformation to the normal cone, gives a deformation of $Y$ into $N_{X/Y}$ over $\A^1$.  Generalizing the case of interest, we suppose that $Y$ is a the total space of a vector bundle $\mc{V}$ on a nonsingular variety $Z$ and that $W$ is induced by a regular section of $\mc{V}^\vee$.  In this situation, the degeneration of $W$ to $N_{X/Y}$ is again a certain tautological function whose theory is easy to control.  We show that the deformation to the normal cone can be used to construct a quasi-equivalence $\DBr(N_{X/Y},\bar{W}) \to \DBr(Y,W)$.  

One of the fundamental properties of matrix factorizations is Kn\"{o}rrer periodicity \cite{Kn} which states that the category of matrix factorizations of a nondegenerate quadratic form $Q$ on an even dimensional vector space is equivalent to the derived category of vector spaces.  Moreover, given an isotropic splitting of the vector space there is a natural realization of this equivalence.  The geometry of $(N_{X/Y},\bar{W})$ is that of a family of nondegenerate quadratic forms and $N_{X/Y}$ has a built in isotropic splitting.  To obtain the final equivalence $\DBr(N_{X/Y},\bar{W}) \he \Perf(X)$ we simply globalize Kn\"{o}rrer's functor and verify that it gives an equivalence.

This paper is organized as follows.  Section \ref{s.const} contains a detailed construction of the category of graded D-branes and discusses the basic operations on these categories and the extra structures these categories enjoy.  In section \ref{s.vb} we consider a general situation of which $(Y,W)$ is example and we prove the main equivalence.  A simple special case of this theorem is a global version of Kn\"{o}rrer periodicity.  Then in section \ref{s.dnb} we discuss the invariance of the category of graded D-branes (up to summands) under deformation to the normal cone.  Section \ref{s.app} brings the results of the previous sections together to obtain a proof of Orlov's theorem.  We also discuss implications for complete intersections.  Finally in \ref{s.loc}, which is somewhat independent from the other sections, we show that the category of D-branes only depends on a formal neighborhood of the critical locus of the potential, when the ambient space is nonsingular and quasi-projective.

\begin{rmk}
There is a strong parallel to this story in the work of Isik \cite{Is}.  He considers the situation where $Y$ is the total space of a vector bundle $\mc{V}$ on a nonsingular variety $Z$ and there is distinguished regular section $s$ of $\mc{V}^\vee$, defining a function $W$ on $Y$.  View $Y_0 = W^{-1}(0)$ as a $\Cx$-scheme via the action by scaling along the fibers of the projection $Y \to Z$.  Isik works with the $\Cx$ equivariant singularity category which is the Verdier quotient ${\sf DSg}^{gr}(Y_0) = D^b\coh^{gr}(Y_0)/\Perf^{gr}(Y_0)$.  Using a version of Koszul duality, he shows that ${\sf DSg}^{gr}(W^{-1}(0))$ is equivalent to $D^b\coh(X)$ where $X$ is the vanishing locus $X = s^{-1}(0) \subset Z$ of the regular section $s$.  Baranovsky and Pecharich \cite{BP} use Isik's theorem to find identifications between graded singularity categories and derived categories of pairs of toric stacks, generalizing a result of Orlov \cite[2.14]{Or}, and interpret their results as a version of the McKay correspondance.
\end{rmk}

\begin{rmk}
The interest in categories of matrix factorizations is partly due to their role in expansions of the Homological Mirror Symmetry program.  Specifically, when $X$ is not Calabi-Yau the mirror object may not be a plain variety but rather a Landau-Ginzburg model consisting of a variety $\hat{X}$ together with a nonconstant regular function $W$ on $X$.  In this picture the category of (graded) matrix factorizations plays the role of the derived category of coherent sheaves on the noncommutative mirror space and should be equivalent to a Fukaya type category.
\end{rmk}

\subsubsection*{Acknowledgements}  My gratitude to Victor Ginzburg for introducing me to graded matrix factorizations and for many wonderful discussions and suggestions.   It was his insight to use the deformation to the normal bundle.  I wish to thank M. Umut Isik, Jesse Burke, and especially Toly Preygel for very interesting discussions.  Furthermore, I greatly appreciate the contribution Ed Segal and Toly Preygel made by pointing out serious mistakes in the various drafts of this article.  Finally, thanks to Igor Dolgachev and Mitya Boyarchenko for bringing my attention to Sumihiro's theorem.

%% file: GO_const.tex
To begin, we will describe the geometric data required to study graded matrix factorizations.  We first discuss group actions and equivariant sheaves.  A variety is a seperated, integral scheme of finite type over $\C$.  Let $\Sigma$ be a variety and $G$ either a torus or a finite group acting on $\Sigma$.  If $\Sigma$ is normal then it admits a $G$-invariant affine open cover, by Sumihiro's theorem \cite{Su}.  Recall that a $(\Cx)^r$ action on $\spec(R)$ is the same as a $\Z^r$ grading on $R$.  Moreover, a quasicoherent sheaf on $[\spec(R)/(\Cx)^r]$ is simply a $\Z^r$ graded module over $R$.  So when $G$ is a torus, we can understand $\Sigma$ through a system of charts where each chart is the spectrum of a $\Z^r$ graded ring.

Suppose that $\chi$ is a character of $G$ and $\C_\chi$ is the corresponding one dimensional representation.  There is a shifting operation corresponding to $\chi$.  If $\cF$ is a $G$-equivariant quasicoherent sheaf on $\Sigma$, then the $\chi$-shift is defined to be $\cF(\chi) = \cF \tensor_\C \C_\chi$.  In the case where $G \cong (\Cx)^r$ acts on $\spec(R)$ a character $\chi$ is described by a vector $\mb{v} \in \Z^r$.  The $\chi$-shift on a $\Z^r$ graded module $M$ is given by the $\mb{v}$ shift so that $M(\chi)_{\mb{u}} = M_{\mb{u} + \mb{v}}$.  Let $\cF_1,\cF_2$ be $G$ equivariant quasicoherent sheaves on $\Sigma$.  For any character $\chi$ of $G$ we refer to a morphism $\cF_1 \to \cF_2(\chi)$ as a morphism $\cF_1 \to \cF_2$ of weight $\chi$.  In particular, morphisms $\O_\Sigma \to \cF(\chi)$ are global sections of $\cF$ of weight $\chi$.  Alternatively, a global section $s$ of $\cF$ has weight $\chi$ if $g^*s = \chi(g)s$ for all $g \in G$.

In this article we work with stacks $S$ of the form $[\Sigma_S/G]$ where $\Sigma_S$ is a variety and $G$ is a torus or a finite abelian group.  The case where $G$ is trivial is an important case.  Recall that a quasicoherent sheaf $\cF$ on $S$ is a $G$-equivariant quasicoherent sheaf on $\Sigma_S$.  By vector bundle on $S$ we mean a $G$-equivariant locally free sheaf of finite rank on $\Sigma_S$.  Global sections of a sheaf on $S$ are the $G$-equivariant sections of the corresponding $G$-equivariant sheaf on $\Sigma_S$.  If $\cF$ is a $G$-equivariant sheaf on $\Sigma_S$ then the equivariant global sections are just the global sections of $\cF$ that are fixed by $G$.

Fix a stack $S = [\Sigma_S/G]$ and put $\Gamma := \Cx$.  The first piece of geometric data that we need is an \emph{even} $\Gamma$ action on $S$.  This is an action of $\Gamma$ on $\Sigma_S$ that commutes with the $G$ action and is such that $\{-1,1\} \subset \Gamma$ acts trivially.  Suppose that $U \cong \spec(R)$ is a $\Gamma \times G$ invariant open affine set.  Then the $\Gamma$ action corresponds to a $\Z$ grading on $R$ and the second condition is equivalent to $R$ being concentrated in even degrees.  We use the notation $\Gamma$ to distinguish this action from the $G$ action.  In the main case of interest for this paper, $G \cong \Cx$.  We refer to the grading induced by $\Gamma$ as the $\Gamma$ grading.  (It is denoted by ``R-charge'' in Segal's work and in the physics literature.)

By \emph{invariant open affine subset of $S$} we mean a $\Gamma \times G$ invariant affine open subset of $\Sigma_S$.  It should be understood that all sheaves and morphisms on such an open affine set are meant to be $G$-equivariant, unless something to the contrary is explicitly stated.

The second piece of geometric data that we need is that of a regular function $F$ on $S$ of $\Gamma$ weight 2.  This is a $G$ invariant regular function on $\Sigma_S$ which has weight 2 for the $\Gamma$ action.  So the restriction of $F$ to any $\Gamma \times G$ invariant open set has $\Gamma$ degree 2.

\begin{dfn}[LG pair]
An \emph{Landau-Ginsburg (LG) pair} $(S,F)$ is a stack $S$ (as above) with an even action of $\Gamma$ and a regular function $F$ on $S$ that has $\Gamma$-weight 2.
\end{dfn}

\begin{exmp}\label{main.example}
In the context of this article, the most important example of an LG pair comes from a variety $Y$ with a vector bundle $\cV$ over it and a section $s$ of $\cV$.  Let $S$ be the total space of $\cV^\vee$, the dual vector bundle and $F$ the function on $\cV^\vee$ that corresponds to $s$.  This function is linear when restricted to each fiber of $S$ over $Y$.   Since $S$ is the total space of a vector bundle it carries a $\Gamma$ action.  However this action is not even, so we ``double'' it by letting $\lambda \in \C$ act by $\lambda^2$.  The local structure of this example is quite simple.  If $U \subset S$ is an open affine subset then $p^{-1}(U)$ is a $\Gamma$ invariant affine set.  Moreover, if $\cV$ is trivialized over $U$ and $p:S \to Y$ denotes the projection then $p^{-1}(U) \cong \spec( \O_Y(U)[y_1,\dotsc,y_r])$ where $y_i$ are the coordinates on $\cV$ given by the trivialization.  The grading assigns $\deg(y_i) = 2$ and $\deg(\O_Y(U)) = 0$.
\end{exmp}

\begin{dfn}[Graded D-brane]
Suppose $(S,F)$ is an LG pair.  A \emph{graded D-brane} on $(S,F)$ is a $\Gamma$-equivariant vector bundle $\cE$, together with an endomorphism $d_{\cE}$ of degree 1 such that $d_\cE^2 = F \cdot \id_\cE$.
\end{dfn}
Recall that a $\Gamma$-equivariant vector bundle on $S$ is a $\Gamma \times G$-equivariant vector bundle on $\Sigma_S$.  Let $U \cong \spec(R_\bt)$ be a $\Gamma$ invariant affine open subset of $S$ and $\cE$ a graded D-brane on $(S,F)$.  Then $\cE(U)$ is simply a graded, projective $R_\bt$-module with an endomorphism that raises degrees by 1 and squares to multiplication by $F$.  (According to our convention, $\cE(U)$ and $d_{\cE}$ must be $G$-equivariant.)  For any $\Gamma$-equivariant sheaf $\cF$ on $S$, let $\sigma$ be the endomorphism induced by the action of $-1 \in \Gamma$.  The action of $\sigma$ on a homogeneous $m \in \cF(U)$ is by $\sigma(m) = (-1)^{\deg(m)}m$.  Let $\cE_1,\cE_2$ be two graded D-branes on $(S,F)$.  We define an endomorphism of $\sHom(\cE_1,\cE_2)$ by $d(\phi) = d_2 \circ \phi - \sigma(\phi) \phi \circ d_1$.  Note that $d^2 = 0$ so the graded $R_\bt$ module $\sHom(\cE_1,\cE_2)(U)$ can be viewed as a complex of DG $R_\bt$ modules, where the differential on $R_\bt$ is zero.  We may define the $\Gamma$-equivariant coherent sheaf
\[ \cH(\sHom(\cE_1,\cE_2)) := \ker(d)/\im(d). \]
We have $\cH(\sHom(\cE_1,\cE_2))(U) = \sfH^\bt( \sHom(\cE_1,\cE_2)(U))$.  As we will see, $\cH(\sHom(\cE_1,\cE_2))$ is always supported on the critical scheme of the zero fiber of $F$.

Let $(S,F)$ be an LG pair.  The \emph{Jacobi ideal (sheaf)} $J(F)$ of $F$ is defined to be the image of the map $\Theta_S \to \O_S$ given by contraction with $dF$, where $\Theta_S$ is the tangent sheaf.  The \emph{Tyurina ideal (sheaf)} is defined to be $\tau(F) := J(F) + F\cdot \O_S$.  If $S$ is nonsingular, the Tyurina ideal sheaf defines the scheme theoretical singular locus of the zero locus of $F$.  Let $Z$ be the subscheme associated to $\tau(F)$.  Observe that $\tau(F)$ is $\Gamma$-equivariant and hence $Z$ is invariant.

Suppose that $\cE$ is a graded D-brane on $(S,F)$.  We can trivialize $\cE$ over a small open set so that $d=d_\cE$ becomes a matrix.  If $v$ is a local vector field we can differentiate the entries of $d$ to obtain an endomorphism $v(d)$.  Now we have
\[ v(F)\cdot \id = v( d^2 ) = d v(d) + v(d) d. \]
Hence multiplication by $v(F)$ is nullhomotopic.  It follows that for any graded D-branes $\cE,\cF$, the cohomology sheaf $\cH(\sHom(\cE,\cF))$ is annihilated by $\tau(F)$ and thus $\cH(\sHom(\cE,\cF))$ is supported on $Z$.

\subsection{\v{C}ech model}

If $\{U_\alpha\}$ is a $\Gamma$-invariant open affine cover, then we put 
\[ \vC^\bt(S,\{U_\alpha\},\sHom(\cE_1,\cE_2)):= \vC^\bt(\Sigma_S,\{U_\alpha\},\sHom(\cE_1,\cE_2))^G_\bt.\]
Note that this is a bicomplex where the first grading is the grading on the \v{C}ech complex (upper bullet) and the second grading is the $\Gamma$ grading on the summands (lower bullet).  From now on we restrict attention to finite $\Gamma$-invariant affine open covers.

If $\cE_1,\cE_2,\cE_3$ are graded D-branes, then the natural composition map $\sHom(\cE_2,\cE_3)\tensor \sHom(\cE_1,\cE_2) \to \sHom(\cE_1,\cE_3)$ is compatible with the differentials.  Hence this map induces a chain map
\begin{equation}\label{dbr.composition}
 \vC^\bt(S,\{U_\alpha\},\sHom(\cE_2,\cE_3)) \tensor_\C \vC^\bt(S,\{U_\alpha\},\sHom(\cE_1,\cE_2)) \to \vC^\bt(S,\{U_\alpha\},\sHom(\cE_1,\cE_3)).
\end{equation}

\begin{dfn}[Category of graded D-branes]
Let $(S,F)$ be an LG pair and $\{U_\alpha\}$ a $\Gamma$-invariant affine open cover of $S$.  The DG category $\DBr(S,F, \{U_\alpha\})$ of graded D-branes is the DG category whose objects are graded D-branes and where
\[ \Hom^\bt_{\DBr(S,F)}(\cE_1,\cE_2) := \vC^\bt(S,\{U_\alpha\},\sHom(\cE_1,\cE_2))_\bt \]
is the total complex of the bicomplex.  Composition in this DG category is given by \eqref{dbr.composition}.
\end{dfn}

The quasi-equivalence class of $\DBr(S,F, \{U_\alpha\})$ does not depend on the specific choice of $\Gamma$-invariant affine open cover.  First, the total complex of the bicomplex $\vC^\bt(S,\{U_\alpha\},\sHom(\cE_1,\cE_2))_\bt$ has a finite filtration
\[ F^i \vC^\bt(S,\{U_\alpha\},\sHom(\cE_1,\cE_2)) := \bigoplus_{j \geq k} \vC^j(S,\{U_\alpha\},\sHom(\cE_1,\cE_2)) \]
which gives rise to a convergent spectral sequence
\begin{equation}\label{main.ss}
E_2^{i,j} = \sfH^i(S, \cH( \sHom(\cE_1,\cE_2) ))_j \Rightarrow \sfH^{i+j}(\vC^\bt(S,\{U_\alpha\},\sHom(\cE_1,\cE_2))).
\end{equation}
The other filtration, while not finite is locally finite since the \v{C}ech degree is bounded and therefore there is another convergent spectral sequence 
\begin{equation}\label{alt.ss}
E_1^{i,j} = \sfH^i(S, \sHom(\cE_1,\cE_2))_j \Rightarrow \sfH^{i+j}(\vC^\bt(S,\{U_\alpha\},\sHom(\cE_1,\cE_2))).
\end{equation}
Now if $\{V_\beta\}$ is another $\Gamma$-invariant open affine cover, then $\{U_\alpha \cap V_\beta\}$ is a common $\Gamma$-invariant open affine refinement.  Moreover, there are comparison maps
\[
\xymatrix{
\vC^\bt(S,\{U_\alpha\},\sHom(\cE_1,\cE_2)) & \ar[l] \vC^\bt(S,\{U_\alpha \cap V_\beta\},\sHom(\cE_1,\cE_2)) \ar[r] & \vC^\bt(S,\{V_\beta\},\sHom(\cE_1,\cE_2))
}
\]
which are compatible with the filtrations by \v{C}ech degree and thus \eqref{main.ss} can be used to show that they are quasi-isomorphisms.  In fact these comparison maps are compatible with \eqref{dbr.composition} as well.  So the comparison maps define quasi-equivalences
\[
 \xymatrix{
\DBr(S,F,\{U_\alpha\}) & \ar[l] \DBr(S,F,\{U_\alpha \cap V_\beta\}) \ar[r] & \DBr(S,F,\{V_\beta\}).
}
\]
From now on, we write $\DBr(S,F)$, suppressing the choice of cover, since the ambiguity in defining the category is rectified by canonical quasi-equivalences.  Denote by $[\DBr(S,F)]$ the homotopy category of $\DBr(S,F)$.  (It has the same objects but the hom space between $\cE_1$ and $\cE_2$ in $[\DBr(S,F)]$ is $\sfH^0 \Hom_{\DBr(S,F)}(\cE_1,\cE_2)$.)

\subsection{Triangulated structure and a second model}

The next goal is to show that $[\DBr(S,F)]$ is triangulated.  Following \cite{LP,Po,Or5} we will define a triangulated category $\DQcoh(S,F)$ which contains $[\DBr(S,F)]$ as a full triangulated subcategory.  We will see that in fact, it is dense in the subcategory of compact objects.  We use a slightly modified version of the notation in \cite{LP} and many ideas from \cite{Po2}.

\begin{dfn}
A \emph{curved graded quasicoherent sheaf} is a pair $(\cF,d)$ where $\cF$ is a $\Gamma$-equivariant quasicoherent sheaf on $S$ and $d$ is an endomorphism of $\cF$ of weight one such that $d^2 = F \cdot \id$.  We denote by $\grQcoh(S,F)$ the category whose objects are curved quasicoherent sheaves and where the complex of morphisms between $(\cF_1,d_1)$ and $(\cF_2,d_2)$ is
\[ \grHom^\bt(\cF_1,\cF_2), \]
the graded space of $\Gamma$-equivariant morphisms of all weights, equipped with the commutator differential.
\end{dfn}

There is a natural shift functor on $[\grQcoh(S,F)]$ and a collection of distinguished triangles that give it the structure of a triangulated category.  Let $\chi_n$ be the $n$-th power character of $\Gamma$.  We define a shift functor on $\grQcoh(S,F)$ by $(\cF,d)[1] := (\cF(\chi_1),-d)$.  There is a natural collection of distinguished triangles as well.  Let $\phi:\cF_1 \to \cF_2$ be a $\Gamma$-equivariant morphism that intertwines the differentials.  Then we define
\[ \cn(\phi) = \left( \cF_2 \oplus \cF_1[1], \begin{pmatrix} d_1 & 0 \\ \phi & -d_2 \end{pmatrix} \right). \]
A distinguished triangle in $[\grQcoh(S,F)]$ is a triangle isomorphic to one of the form
\[ \cF_1 \stackrel{\phi}{\to} \cF_2 \to \cn(\phi) \to \cF_1[1] \to \dotsm \]
The shift and triangles define a triangulated structure on $[\grQcoh(S,F)]$ like their analogs in the case of complexes.

Suppose that 
\[ \to \dotsm \stackrel{\delta_{-2}}{\to} \cF_{-1} \stackrel{\delta_{-1}}{\to} \cF_0 \stackrel{\delta_0}{\to} \dotsm \stackrel{\delta_{n-1}}{\to} \cF_n \stackrel{\delta_{n}}{\to} \dotsm \]
is a complex of curved graded quasicoherent sheaves, meaning that each $\delta$ is $\Gamma$ equivariant and intertwines the differentials.  The \emph{totalization} of this complex is the curved graded quasicoherent sheaf whose underlying sheaf is
\[ \bigoplus_{i \in \Z} {\cF_i(\chi_{-i})}, \]
and where the restriction of the differential to the direct factor $\cF_i(\chi_{-i})$ is $(-1)^i d_i \oplus \delta_i$, which has values in the direct factor $\cF_i(\chi_{-i}) \oplus \cF_{i+1}(\chi_{-(i+1)})$.

The category $\grQcoh(S,F)$ contains a natural subcategory of acyclic objects.  Let $[\grQcoh(S,F)]_{ac}$ be the thick triangulated subcategory generated by the totalizations of bounded acyclic complexes of curved graded quasicoherent sheaves and closed under arbitrary direct sums.

\begin{dfn}
The \emph{derived category of curved graded quasicoherent sheaves} is the Verdier quotient
\[ \DQcoh(S,F) = [\grQcoh(S,F)]/[\grQcoh(S,F)]_{ac}. \]
We also write $\grDcoh(S,F)$ for the full subcatogory of $\DQcoh(S,F)$ of objects isomorphic to coherent curved graded sheaves.
\end{dfn}

\begin{rmk} This category is often called the ``absolute derived category''.  However, since it will not face any competition from rival definitions in this article, we simply call it the ``derived category''.
\end{rmk}

Clearly we can view a graded D-brane as an object in $\DQcoh(S,F)$.  Let $\MF(S,F) \subset [\Qcoh(S,F)]$ and $\DMF(S,F) \subset \DQcoh(S,F)$ be the full subcategories whose objects are graded D-branes.  Observe that $\DMF(S,F)$ is the Verdier quotient $\MF(S,F)/\MF(S,F) \cap [\grQcoh(S,F)]_{ac}$.  Clearly, $\DMF(S,F)$ is triangulated.

\begin{rmk}
The triangulated category $\DMF(S,F)$ differs from the absolute derived category of graded matrix factorizations appearing in \cite{BW} in that $\MF(S,F) \cap [\grQcoh(S,F)]_{ac}$ is potentially larger than the category generated by acyclic objects built from finitely many graded D-branes.
\end{rmk}

Fix a $\Gamma$ invariant open cover $\{U_\alpha\}_{\alpha \in A}$ of $S$.  The \v{C}ech complex of a sheaf $\cF$ has a sheaf theoretic analogue.  Put
\[ \svC^i(\cF) = \bigoplus_{\substack{ A' \subset A \\ |A'| = i }}{ (j_{A'})_* j^*_{A'}\cF } \]
where $j_{A'}$ is the inclusion of the open set $\bigcap_{\alpha \in A'} U_\alpha$.  The differentials are defined by the familiar formula.  Note that the usual \v{C}ech complex is obtained by taking global sections of this complex.  If $\cF$ is $\Gamma$-equivariant, then $\svC^\bt(\cF)$ is a $\Gamma$-equivariant complex of sheaves.

This complex of sheaves is functorial in $\cF$, so if $\cF$ is a curved graded quasicoherent sheaf then for each $i$, $\svC^i(\cF)$ has a natural differential making it into a curved graded quasicoherent sheaves.  We write $\svC(\cF)$, without the bullet, for the totalization of $\svC^\bt(\cF)$.  There is a canonical morphism
\[ \grHom^\bt(\cF_1,\cF_2) \to \grHom^\bt(\svC(\cF_1),\svC(\cF_2)) \]
which is compatible with compositions, so we can think of $\svC$ as a DG endofunctor of $\grQcoh(S,F)$.  Note that the functor induced by $\svC$ on $[\grQcoh(S,F)]$ preserves the subcategory $[\grQcoh(S,F)]_{ac}$.

Let $\cF_1,\cF_2$ be curved graded quasicoherent sheaves.  Let us use the short hand
\[ \vC^\bt(\cF_1,\cF_2) := \vC^\bt(S,\{U_\alpha\},\sHom(\cF_1,\cF_2))_\bt. \]
Observe that the canonical morphism
\[  \vC^\bt(\cF_1,\cF_2) \to \grHom^\bt(\cF_1,\svC(\cF_2)) \]
is an isomorphism of complexes.  Suppose that $\cF_1$ is a graded D-brane and $\cF_2$ is the totalization of an acyclic complex of curved graded quasicoherent sheaves.  Then $\sHom(\cF_1,\cF_2)$ is the totalization of an acyclic complex of 0-curved graded quasicoherent sheaves.  It has an extra grading in addition to the $\Gamma$-equivariant structure.  Thus $\vC^\bt(\cF_1,\cF_2)$ is a bicomplex with the usual grading and the extra grading coming ultimately from the extra grading on $\cF_2$.  There is a convergent spectral sequence computing the cohomology of $\vC^\bt(\cF_1,\cF_2)$ whose zero-th page has only the differential induced by the differential on $\sHom(\cF_1,\cF_2)$ with respect to which it is acyclic.  Since taking the \v{C}ech complex is exact we see that the first page of this spectral sequence is zero and thus
\[ \sfH^0\grHom^\bt(\cF_1,\svC(\cF_2)) = 0. \]
Therefore, the functor $\sfH^0 \grHom^\bt(\cF_1,\svC(-))$ on $[\grQcoh(S,F)]$ factors through $\DQcoh(S,F)$.  The natural map $\cF \to \svC(\cF)$ is an isomorphism in $\DQcoh(S,F)$ since the \v{C}ech complex resolves $\cF$ and thus the cone on $\cF \to \svC(\cF)$ is the totalization of an acyclic complex of curved graded quasicoherent sheaves.  So there is an obvious natural transformation
\begin{equation}\label{cech-dqcoh} \sfH^0\grHom^\bt(\cF_1,\svC(-)) \to \Hom_{\DQcoh(S,F)}(\cF_1,-). \end{equation}
We shall briefly check that this is an isomorphism.  Recall that $\Hom_{\DQcoh(S,F)}(\cF_1,\cF_2)$ consists of equivalence classes of pairs of morphisms
$$ \cF_1 \stackrel{g}{\to} \cG \stackrel{f}{\from} \cF_2 $$
where the cone on $f$ belongs to $[\grQcoh(S,F)]_{ac}$.  Note that the map induced by $f$,
$$ \sfH^0\grHom^\bt(\cF_1,\svC(\cF_2)) \to \sfH^0\grHom^\bt(\cF_1,\svC(\cG)), $$
is an isomorphism because $\sfH^\bt\grHom(\cF_1,\svC(\cn(f)))=0$.  Thus, the map $g:\cF_1 \to \cG$ factors through $f$ and we see that \ref{cech-dqcoh} is surjective.  Next, suppose that a closed morphism $\phi:\cF_1 \to \svC(\cF_2)$ maps to zero in $\Hom_{\DQcoh(S,F)}(\cF_1,\cF_2)$.  By definition of the equivalence relation defining homs in $\DQcoh(S,F)$ there must be a diagram
$$ \xymatrix{
& \svC(\cF_2) \ar[d] & \\
\cF_1 \ar[ur]^\phi \ar[r] \ar[dr]_0 & \cG & \ar[ul] \ar[l]^{f} \ar[dl]^{\id} \cF_2 \\
& \cF_2 \ar[u] &
}
$$
where once again, the cone on $f$ is an acyclic curved quasicoherent sheaf.  The morphisms on the right side of the diagram all induce isomorphisms after applying $\sfH^0\grHom^\bt(\cF_1,\svC(-))$.  Of course, this means that $\phi$ and zero are identified and thus $\phi = 0$.  We conclude that the space of morphisms between $\cF_1$ and $\cF_2$ in $\DQcoh(S,F)$ is computed by $\vC^\bt(\cF_1,\cF_2)$.

There is a canonical DG functor
\[ \svC: \DBr(S,F) \to \grQcoh(S,F) \]
since the embedding
\[ \vC^\bt(\cE_1,\cE_2) \to \grHom^\bt( \svC(\cE_1), \svC(\cE_2)) \]
is compatible with composition.  It follows now that the induced functor
\[ \svC: [\DBr(S,F)] \to \DQcoh(S,F) \]
is fully faithful.  Thus we have the following.
\begin{prop}
$[\DBr(S,F)]$ is equivalent to $\DMF(S,F)$ and thus it is triangulated.
\end{prop}

We need a few ideas from the theory of triangulated categories.  Let $T$ be a triangulated category.  An object $E \in T$ is \emph{compact} if for any collection of objects $F_i \in T$ the natural map
\begin{equation}\label{sum-comparison} \bigoplus_i \Hom_T(E,F_i) \to \Hom_T(E,\bigoplus_i F_i) \end{equation}
is an isomorphism (provided $\bigoplus_i F_i$ exists).  The full subcategory of compact objects $T^c$ of $T$ is a triangulated subcategory, closed under taking direct summands.  If $T' \subset T$ is a triangulated subcategory, then $T'$ is \emph{dense} if every object of $T$ is a summand of an object in $T'$.  We say that $T$ has \emph{direct sums} if $T$ admits all small direct sums and direct sums of triangles are triangles.  We say that a set $\{E_i\}$ \emph{weakly generates} if whenever $\Hom_T(E_i,E')=0$ for all $i$ then $E'=0$.  $T$ is \emph{compactly generated} if $T$ has direct sums and there is a set of compact objects that weakly generates $T$.  According to \cite{N} if $T$ has direct sums and a set $\{E_i\}$ of compact objects weakly generates, then $T^c = \langle E_i \rangle$.  Here, $\langle E_i \rangle$ is the smallest full triangulated subcategory containing $\{E_i\}$ and closed under taking shifts, cones, isomorphic objects, and direct summands.

A triangulated category is Karoubi complete if whenever $\Psi \in \End(E)$ is an idempotent, there is an object $E'$ and morphisms $\psi_1:E \to E'$ and $\psi_2:E' \to E$ such that $\Psi = \psi_2 \circ \psi_1$ and $\psi_1 \circ \psi_2 = \id$.  If $T$ is an arbitrary triangulated category, we write $\ol{T}$ for the Karoubi completion.  This is the triangulated category obtained by formally adding objects to $T$ to make it Karoubi complete.

\begin{lma}\label{compact}
Graded D-branes are compact objects in $\DQcoh(S,F)$.
\end{lma}
\begin{proof}
Let $\cE$ be a graded D-brane. $\Gamma$-equivariant vector bundles satisfy the compactness property in the category of all $\Gamma$-equivariant quasicoherent sheaves.  Thus for any collection $\{\cF_i\}$ of curved graded quasicoherent sheaves, the natural map
\[ \bigoplus_i \vC^\bt(\cE,\cF_i) \to \vC^\bt( \cE, \bigoplus_i \cF_i ) \]
is an isomorphism of complexes.  Since $\vC^\bt(\cE,\cF)$ computes $\Hom_{\DQcoh(S,F)}(\cE,\cF)$ and the isomorphism \eqref{cech-dqcoh} is compatible with the comparison map \eqref{sum-comparison}, the result follows.
\end{proof}

\begin{dfn}
We say that $S$ is \emph{$\Gamma$-quasiprojective} if it admits a $\Gamma$-equivariant ample line bundle.
\end{dfn}

\begin{lma}\label{generation}
Suppose $S$ is $\Gamma$-quasiprojective and $\Sigma_S$ is nonsingular.  Let $\cF \in \DQcoh(S,F)$.  If
\[ \Hom_{\DQcoh(S,F)}(\cE,\cF)=0 \]
for all graded D-branes $\cE$ then $\cF = 0$.
\end{lma}
\begin{proof}
First, we will show that every coherent curved graded sheaf $\cF$ is a quotient of a graded D-brane.  Suppose that $\cL$ is the $\Gamma$-equivariant ample line bundle and write $\cF(m) := \cF \tensor \cL^{\tensor m}$.  If $m$ is sufficiently large, $\cF(m)$ is globally generated.  Let $R$ be the ring of global functions on $S$ and $M$ the $R$ module of global sections of $\cF(m)$.  To express $\cF$ as a quotient of a graded D-brane, it suffices to express $M$ as a quotient of a graded D-brane over $(R,F)$.  There is a trivial way to do this.  We think of $M$ as a graded $R' := R[t]/(t^2 - F)$ module, where $t$ has degree 1.  Then since $M$ is finitely generated as an $R$ module, it is finitely generated as an $R'$ module.  Thus there is a surjection $(R')^{\oplus N} \onto M$ of graded $R'$ modules.  Thinking of the action of $t$ on $R'$ as an $R$ module endomorphism, we can identify $R'$ with the trivial graded D-brane
\[ \xymatrix{ R \ar@/^/[r]^{\id} & \ar@/^/[l]^F R }. \]
Since $R$ is the ring of global functions there is a $\Gamma$-equivariant morphism $S \to \spec(R)$.  Pulling back to $S$, we obtain a graded D-brane $\cE$ and a surjection $\cE \onto \cF(m)$ which we easily transform into the desired form $\cE(-m) \onto \cF$.

Let $\cF$ be a coherent curved graded sheaf.  For each $n$, there is a partial resolution
\[ \cE_n \to \cE_{n-1} \to \dotsm \cE_0 \onto \cF \]
where $\cE_i$ is a graded D-brane.  If $n > \dim \Sigma_S$ then $\ker( \cE_n \to \cE_{n-1})$ is automatically a vector bundle since $\Sigma_S$ is smooth.  (While the $\cE_i$ are trivial we don't know anything about this final kernel beyond the fact that it is a graded D-brane.)  Thus, we obtain a resolution for $\cF$ whose terms are graded D-branes.  The totalization of this complex (without $\cF$) is a graded D-brane which is isomorphic to $\cF$ in $\DQcoh(S,F)$.  Hence, $\DMF(S,F)=\Dcoh(S,F)$.

Let $\cF$ be arbitrary such that $\Hom(\cE,\cF) = 0$ for every graded D-brane $\cE$.  Then $\Hom(\cE,\cF)=0$ whenever $\cE$ is a coherent curved graded sheaf, since every coherent curved graded sheaf is isomorphic is $\DQcoh(S,F)$ to a graded D-brane. The rest of the argument can be adapted mutatis mutandis from the work of Positselski. He proves the analogous statement as Theorem 2 in Section 3.11 of \cite{Po2}.  The arguments given for Theorem 2 in Section 3.11 (including the supporting Theorem 3.6 and the Theorem of Section 3.7) use only properties of and constructions in categories of graded modules.  These arguments go through in this case because the category of $\Gamma \times G$-equivariant sheaves on $S$ has enough injectives and satisfies Positselski's condition (*) since $\Sigma_S$ is Noetherian.
\end{proof}

Clearly $\DQcoh(S,F)$ has arbitrary direct sums and the direct sum of triangles in $\DQcoh(S,F)$ is a triangle.  Now, Lemmas \ref{compact} and \ref{generation} show that $\DQcoh(S,F)$ is compactly generated.  It follows from \cite{BN} that the subcategory $\DQcoh(S,F)^c$ of compact objects is Karoubi complete (all idempotents split).  So we have the following.

\begin{prop}\label{Karoubian-completion}
If $S$ is $\Gamma$-quasiprojective and $\Sigma_S$ is nonsingular, then
$$\ol{[\DBr(S,F)]} \cong \DQcoh(S,F)^c.$$
\end{prop}

\begin{rmk}
Observe that in the case where the $\Gamma$ action on $S$ is trivial, fixing all of $S$, we have $\DBr(S,0) = \Perf(S)$.  So if $S$ is nonsingular then $[\DBr(S,0)] \cong \Dcoh(S)$.  
\end{rmk}

\begin{rmk}\label{flat-branes}
It follows from \ref{generation} by the results of \cite{N} that $\DQcoh(S,F)$ is equivalent to the smallest triangulated subcategory of $\DQcoh(S,F)$ containing all graded D-branes and closed under taking direct sums.  Hence, every object of $\DQcoh(S,F)$ is isomorphic to an object whose underlying $\Gamma$-equivariant sheaf is flat.
\end{rmk}

\subsection{Pullback and external tensor product}

There is often a natural pullback functor associated to a morphism of LG pairs.  Suppose that $(S_1,F_1)$ and $(S_2,F_2)$ are two LG pairs and that $\phi:S_1 \to S_2$ is a $\Gamma$-equivariant morphism such that $F_1 = \phi^*F_2$.  Let $\{U_\alpha\}$ be a $\Gamma$-invariant affine open cover of $S_2$ and $\{V_\beta\}$ a $\Gamma$-invariant affine open cover refining $\{\phi^{-1}(U_\alpha)\}$.  Note that if $\cE$ is a graded D-brane on $(S_2,F_2)$ then $\phi^*\cE$ is a graded D-brane on $(S_1,F_1)$.  In addition there are chain maps
\[ \vC^\bt(S_2,\{U_\alpha\},\sHom(\cE_1,\cE_2)) \to \vC^\bt(S_1,\{V_\beta\},\sHom(\phi^*\cE_1,\phi^*\cE_2)) \]
which are compatible with compositions.  So there is a DG functor
\[ \phi^*: \DBr(S_2,F_2) \to \DBr(S_1,F_1) \]
that is compatible with the cone construction and thus induces an exact functor
\[ [\DBr(S_2,F_2)] \to [\DBr(S_1,F_1)]. \]

We can extend the notion of pullback to $\DQcoh(S_2,F_2)$ in the case that $S_2$ satisfies the hypotheses of \ref{generation}.  We simply define $\phi^*$ to be the naive pullback on the subcategory of $\DQcoh(S_2,F_2)$ consisting of objects whose underlying sheaves are flat.  By Remark \ref{flat-branes}, this subcategory is equivalent to $\DQcoh(S_2,F_2)$.

There is one last type of functorial construction that we can consider, the tensor product.  Let $S$ be an even graded scheme or stack and $F_1,F_2$ two semi-invariant regular functions of degree 2.  Then there is a natural DG functor $\DBr(S,F_1) \tensor_\C \DBr(S,F_2) \to \DBr(S,F_1+F_2)$.  If $\cE_1$ and $\cE_2$ are graded D-branes on $(S,F_1)$ and $(S,F_2)$ respectively and we define an endomorphism of $\cE_1 \tensor \cE_2$ by
\[ d = d_1 \tensor \id + \sigma \tensor d_2 \]
then $(\cE_1\tensor \cE_2, d)$ is a graded D-brane on $(S,F_1+F_2)$.  (Recall that $\sigma(m) = (-1)^{\deg(m)}m$ when $m$ is homogeneous.)  Now, if $\cE_1,\cF_1$ and $\cE_2,\cF_2$ are graded D-branes on $(S,F_1)$ and $(S,F_2)$ respectively, then there is a canonical isomorphism
\[ \sHom(\cE_1\tensor \cE_2,\cF_1 \tensor \cF_2) \cong \sHom(\cE_1,\cF_1) \tensor \sHom(\cE_2,\cF_2) \]
and therefore a chain map
\[
 \vC^\bt(S,\{U_\alpha\},\sHom(\cE_1,\cF_1)) \tensor_\C \vC^\bt(S,\{U_\alpha\},\sHom(\cE_2,\cF_2)) \to \vC^\bt(S,\{U_\alpha\},\sHom(\cE_1\tensor \cE_2,\cF_1 \tensor \cF_2)).
\]
This is compatible with the composition chain maps and therefore we obtain a DG functor
\[ \DBr(S,F_1)\tensor_\C \DBr(S,F_2) \to \DBr(S,F_1+F_2). \]

%% file: GO_vb.tex
In this section we study in detail the LG pair introduced in Example \ref{main.example}.  Let $Y$ be a nonsingular quasiprojective variety and $\cV$ a vector bundle over $Y$.  Suppose that $s$ is a regular section of $\cV$, meaning that the rank of $\cV$ is the same as the codimension of the zero locus of $s$.  Denote by $Z$ the zero locus of $s$.  Write $S$ for the total space of the vector bundle $\cV^\vee$ and let $p:S \to Y$ be the projection.  We can view $s$ as a regular function $F$ on $S$ that is linear on the fibers of $p$.  Since $S$ is the total space of a vector bundle, it has a natural $\Cx$ action.  However, we consider the LG pair $(S,F)$ where the $\Gamma$ action is obtained by ``doubling'' the $\Cx$ action via the squaring map $\lambda \mapsto \lambda^2$.

There is a distinguished object in $\DBr(S,F)$, which we denote by $\cK$.  To construct it, we begin by observing that $p^*\cV^\vee$ has both a canonical section and a canonical cosection.  The cosection, of course, is the pullback of $s$.  The section $s_Y$, on the other hand, is the canonical section of $p^*\cV^\vee$ which vanishes precisely on the zero section $Y \subset S$.  The composition
$$ \O_S \stackrel{s_Y}{\to} p^*\cV^\vee \stackrel{s}{\to} \O_S $$
is simply $F$. Now, since $p:S \to Y$ is $\Gamma$ invariant, $p^*\cV^\vee$ has a canonical $\Gamma$-equivariant structure.  With respect to this equivariant structure, $s^\vee$ is equivariant while $s_Y$ has weight two.  Hence $s$ and $s_Y$ have weight one as morphisms $\O_S \to p^*\cV^\vee(\chi_{-1})$ and $p^*\cV^\vee(\chi_{-1}) \to \O_S$, respectively.  (Recall from Section \ref{s.const} that $\cF(\chi)$ has equivariant structure twisted by the character $\chi$. In this case, shifting sections from weight zero into weight one.) We can now define $\cK$.

\begin{dfn}[Distinguished graded D-brane]
Let $\cK \in \DBr(S,F)$ be the graded D-brane whose underlying $\Gamma$-equivariant vector bundle is
$$ \bigwedge^\bt p^*\cV^\vee(\chi_{-1}) $$
and whose differential is given by the formula
$$ d(-) = s_Y \wedge (-) + s \vee (-). $$
\end{dfn}

This graded D-brane has a very important property.  View $\wedge^\bt p^*\cV^\vee(\chi_{-1})$ as a $\Gamma$-equivariant sheaf of algebras.  Then left multiplication gives a map
$$ \wedge^\bt p^*\cV^\vee(\chi_{-1}) \to \sEnd(\cK). $$
On the other hand, contraction gives a map
$$ \wedge^\bt p^* \cV(\chi_1) \to \sEnd(\cK). $$
Together, these two maps give an isomorphism
$$ \bigwedge^\bt \left( p^* \cV(\chi_1) \oplus p^*\cV^\vee(\chi_{-1}) \right) \cong \wedge^\bt p^* \cV(\chi_1) \tensor \wedge^\bt p^*\cV^\vee(\chi_{-1}) \to \sEnd(\cK). $$
It is straightforward to check that the isomorphism intertwines the differential on $\sEnd(\cK)$ with the Koszul differential
$$ d = (s_Y \oplus s) \vee -. $$
Hence there is a quasi-isomorphism
$$ \sEnd(\cK) \to \O_Z $$
since $Z$ is the vanishing locus of $s_Y \oplus s$, viewed as a section of $p^* \cV^\vee \oplus \cV$ and $s_Y \oplus s$ is a regular section.

Next, we define $\Sbr = p^{-1}Z \subset S$ and let $i:\Sbr \to S$ be the inclusion.  Since $F$ vanishes on $\Sbr$ and $\Sbr$ is $\Gamma$-invariant, we may view $i_*\O_{\Sbr}$ as an object in $\grQcoh(S,F)$.  Note that
$$ ( \wedge^\bt p^* \cV^\vee, s\vee- ) $$
is a resolution of $i_* \O_{\Sbr}$.  The natural map
$$ \cK \to i_*\O_{\Sbr} $$
actually sends the differential on $\cK$ to zero and therefore gives a morphism in $\grQcoh(S,F)$.

\begin{lma}\label{Koszul}
The morphism $\cK \to i_*\O_{\Sbr}$ is an isomorphism in $\DQcoh(S,F)$.  
\end{lma}
\begin{proof}
It suffices to check that the cone on this morphism is zero.  We will prove a more general statement that is more convenient to work with.  Suppose that $\cC^\bt$ is a bounded $\Z$ graded $\Gamma$ equivariant sheaf with endomorphisms $\alpha$ and $\beta$ that have $\Gamma$-weight one and raise and lower the bullet grading, respectively.  Assume that $\alpha^2 = \beta^2 = 0$ and $\alpha \circ \beta + \beta \circ \alpha = F \cdot \id$ so that $(\cC^\bt,\alpha+\beta)$ is a curved, graded quasicoherent sheaf.  If the complex $(\cC^\bt,\alpha)$ is acyclic then $(\cC^\bt,\alpha+\beta)$ is zero in $\DQcoh(S,F)$.

We proceed by induction on the number of bullet-degrees in which $\cC^\bt$ is nonzero.  In order for $(\cC^\bt,\alpha)$ to be acyclic, $\cC^\bt$ is either zero or has two or more nonzero homogeneous components.  If there are two homogeneous components then $\alpha$ is an isomorphism and therefore $\beta$ is $F \cdot \alpha^{-1}$.  In this case the identity morphism of $(\cC^\bt,\alpha+\beta)$ is nullhomotopic so it is isomorphic to the zero object.

In general, suppose that $n$ is the smallest integer for which $\cC^n \neq 0$.  Observe that $\alpha:\cC^n \to \cC^{n+1}$ must be injective.  So $\alpha$ gives an isomorphism between $\cC^n$ and $\alpha(\cC^n) \subset \cC^{n+1}$.  Now $\cC^n \oplus \alpha(\cC^n) \subset \cC^\bt$ is clearly closed under both $\alpha$ and $\beta$ separately.  Hence $\alpha$ and $\beta$ descend to $\ol{\cC}^\bt := \cC^\bt/(\cC^n \oplus \alpha(\cC^n))$ and satisfy the same relations.  Moreover, by construction $(\ol{\cC}^\bt,\alpha)$ is still acyclic. By induction, it is trivial.  Clearly $(\cC^n\oplus \alpha(\cC^n),\alpha+\beta)$ is trivial and since $(\cC^\bt,\alpha+\beta)$ fits into a triangle with two trivial objects, it is trivial as well.

Finally, we observe that the cone on $\cK \to i_*\O_{\Sbr}$ has the form considered in the preceding paragraphs, with $\alpha$ corresponding to $s^\vee \vee -$ and $\beta$ corresponding to $s_Y \wedge$ and where $\wedge^i p^*\cV^\vee(\chi_{-1})$ in bullet-degree $-i$.
\end{proof}

\begin{rmk}
The Lemma above is a special case of the stabilization discussed by Eisenbud in Section 7 of \cite{E}.
\end{rmk}

Suppose that $Y$ has an ample line bundle $\O_Y(1)$.  Then $S$ is $\Gamma$-quasiprojective with ample bundle $\O_S(1) = p^*\O_Y(1)$.  We put $\cK(i) = \cK \tensor \O_S(i)$, where the tensor product is as $\Gamma$-equivariant vector bundles.

We will now prove that $\DMF(S,F)$ is equivalent to $\Dcoh(Z)$.  First, we note that $i:\Sbr \to S$ is affine and $p:\Sbr \to Z$ is flat.  Therefore, $i_*$ and $p^*$ are exact.  Now $i$ and $p$ are $\Gamma$-equivariant (where $Z$ has the trivial $\Gamma$-equivariant structure) and $F$ vanishes on $\Sbr$.  Therefore there are natural triangulated DG functors
$$ p^*:\grQcoh(Z,0) \to \grQcoh(\Sbr,0), \quad i_*:\grQcoh(\Sbr,0) \to \grQcoh(S,F) $$
and since $i_*$ and $p^*$ are exact on the abelian categories of quasicoherent sheaves and preserve direct sums, they descend to triangulated functors
$$ p^*:\DQcoh(Z,0) \to \DQcoh(\Sbr,0), \quad i_*:\DQcoh(\Sbr,0) \to \DQcoh(S,F). $$
Clearly $p^*$ sends $\grDcoh(Z,0) = \Dcoh(Z)$ to $\grDcoh(\Sbr,0)$ and since $S$ is nonsingular and $\Gamma$-quasiprojective $i_*$ sends $\grDcoh(\Sbr,0) \to \DMF(S,F)$, as in \ref{generation}.

\begin{thm} \label{main-equivalence}
The functor $i_* \circ p^*:\Dcoh(Z) \to \DMF(S,F)$ is an equivalence.
\end{thm}
\begin{proof}
We will first prove that the functor is fully faithful.  Fix an open affine cover $\{U_{Y,\alpha}\}$.  This gives rise to open affine covers $\{ U_\alpha \cap Z \}, \{ p^{-1}U_\alpha \},$ and $\{p^{-1}U_\alpha \cap \Sbr \}$ of $Z,S,$ and $\Sbr$ respectively.  In what follows we use this system of compatible covers to form all \v{C}ech complexes and sheaves, in particular those used implicitly to define categories of graded D-branes.

We need to calculate what the functor does to spaces of morphisms.  First, recall that we can lift $i_*p^*$ to $\grQcoh(S,F)$ by sending $\cP \in \DBr(Z,0)$ to $\svC(i_*p^*\cP)$.  Then there is a natural morphism of complexes
$$ \vC( \sHom(\cP,\cQ)) \to \Hom( \svC(i_*p^*\cP),\svC(i_*p^*\cQ)). $$
Next, we know that the morphism $\cK(j) \to i_*p^*\O_Z(j)$ is an isomorphism in $\DQcoh(S,F)$.   So we consider the arrow
$$
\vC( \sHom(\O_Z(j),\O_Z(j')) \to \Hom( \svC(i_*p^*\O_Z(j)),\svC(i_*p^*\O_Z(j'))) \to \Hom( \svC(\cK(j)),\svC(i_*p^*\O_Z(j')))
$$
and note that the right hand complex computes $\Hom(i_*p^*\O_Z(j),i_*p^*\O_Z(j'))$ in $\DQcoh(S,F)$.  To see that this arrow is a quasi-isomorphism we note that it fits into the commutative triangle
$$
\xymatrix{
\vC( \sHom(\O_Z(j),\O_Z(j'))) \ar[r] & \Hom( \svC(\cK(j)), \svC(i_*p^*\O_Z(j'))) \\
\vC( \sHom( \cK(j),i_*p^*\O_Z(j'))) \ar[ur] \ar[u] &
}
$$
and observe that $\sHom(\cK(j),i_*p^*\O_Z(j')) \cong \cK^\vee|_{\Sbr}(j'-j)$ is none other than (the direct image in $S$ of) the Koszul resolution of $\O_Z(j-j')$ on $\Sbr$.  Since $[\Perf(Z)] = \langle \O_Z(i) \rangle_{i \in \Z}$ (see \cite{BvdB}), this means that $i_*p^*$ is fully faithful on $[\Perf(Z)]$.  (Recall that $\langle \O_Z(i) \rangle_{i \in \Z}$ is the smallest full triangulated subcategory of $\bD(\Qcoh(Z))$ containing $\{\O_Z(i)\}$ and closed under cones, shifts, isomorphic objects, and direct summands.)  We denote by $\langle \O_Z(i) \rangle_{i \in \Z}^\oplus$ the closure of $\langle \O_Z(i) \rangle_{i \in \Z}$ under taking small direct sums and note that by results in \cite{N}, $\Dcoh(Z) \subset \langle \O_Z(i) \rangle_{i \in \Z}^\oplus$.  Since $i_*p^*$ preserves direct sums it follows that $i_* p^*$ is fully faithful on $\Dcoh(Z)$ as well.

Now $\Dcoh(Z)$ is idempotent complete.  Therefore, if the image of $i_*p^*$ weakly generates $\DQcoh(S,F)$ then $i_*p^*$ gives an equivalence $\Dcoh(Z) \to \DQcoh(S,F)^c$ that factors though the inclusion $\DMF(S,F)$.  Since $\DMF(S,F) \subset \DQcoh(S,F)^c$ is dense this means that $i_*p^*$ gives an equivalence between $\Dcoh(Z)$ and $\DMF(S,F)$.

It remains to show that the image of $i_*p^*$ weakly generates $\DQcoh(S,F)$.  Suppose that $\cF \in \DQcoh(S,F)$ is an object such that $\Hom(i_*p^*\cG,\cF)=0$ for all $\cG \in \Dcoh(Z)$.  Let us say that such an object is in the right orthogonal to $i_*p^*$ and write $i_*p^* \perp \cF$.  I claim that if $\cW$ is a $\Gamma$-equivariant vector bundle on $S$ then $i_*p^* \perp \cF \tensor \cW$.  First, note that $\cW|_{\Sbr} \cong p^* \cW|_Z$.  Next for $\cG \in \Dcoh(Z)$, suppose that $\cE \to i_*p^*\cG$ is an isomorphism where $\cE$ is a graded D-brane.  Then
$$\cE \tensor \cW^\vee \to i_*p^*\cG \tensor \cW^\vee = i_*p^*(\cG \tensor \cW|_Z)$$
is also an isomorphism.  Now
$$ \Hom(i_*p^*\cG,\cF \tensor \cW) \cong \Hom( \cE, \cF \tensor \cW) \cong \Hom( \cE \tensor \cW^\vee, \cF) \cong \Hom(i_*p^*(\cG\tensor \cW^\vee|_Z), \cF) = 0 $$
Since $\DMF(S,F)$ weakly generates $\DQcoh(S,F)$ we can assume that the sheaf underlying $\cF$ is a direct sum of equivariant vector bundles and hence flat.  Let $K(s)$ be the Koszul resolution of $\Sbr \subset S$.  It is $\Gamma$-equivariant and $K(s)\tensor \cF \to \cF|_{\Sbr}$ is an isomorphism.  Since $K(s)\tensor \cF$ is an iterated cone on objects of the form $\cF \tensor \cW$ where $\cW=\wedge^m \cV^\vee$ for values of $m$ we see that $i_*p^* \perp \cF|_{\Sbr}$.  Clearly, this implies that $\cF|_{\Sbr} =0$.

Thus for any graded D-brane $\cE$, $\Hom(\cE,K(s) \tensor \cF) = 0$.  We can compute $\Hom(\cE,K(s)\tensor\cF)$ by first pushing $K(s) \tensor \sHom(\cE,\cF)$ down to $Y$, then taking hypercohomology.  Now $p_*K(s) \tensor \sHom(\cE,\cF) = K_Y(s) \tensor p_*\sHom(\cE,\cF)$ and since $p_*\sEnd(\cE)$ is supported on $Z$, so is $p_*\sHom(\cE,\cF)$.  (We write $K_Y(s)$ for the Koszul complex of $s$ on $Y$.)  Observe that this means that $\cH^\bt(K_Y(s) \tensor p_*\sHom(\cE,\cF))$ is locally isomorphic to $\wedge^\bt\cV^\vee \tensor \cH^\bt(p_*\sHom(\cE,\cF))$.  We conclude that $p_*\sHom(\cE,\cF)$ is acyclic if and only if $K_Y(s)\tensor p_*\sHom(\cE,\cF)$ is acyclic.  If the hypercohomology of $K_Y(s)\tensor p_*\sHom(\cE,\cF)(i)$ vanishes for all $i$ then $K_Y(s)\tensor p_*\sHom(\cE,\cF)$ and thus, $p_*\sHom(\cE,\cF)$, has to be acyclic.  However if $p_*\sHom(\cE,\cF)$ is acyclic then $\Hom(\cE,\cF)=0$.  Because $\cE$ was arbitrary, this implies, at last, that $\cF = 0$.
\end{proof}

\begin{rmk}
Orlov uses a similar geometric picture to prove Theorem 2.1 in \cite{Or2}, which relates the category of singularities of the zero locus of a regular section of a vector bundle to the category of singularities of a divisor on the associated projective space bundle.  It is likely that one could adapt his technique to this situation.  However, it has already appeared several times in the literature and therefore we have presented a different approach.
\end{rmk}

\begin{rmk} \label{kp-remark}
Suppose that $Z \subset Y$ is actually a retract, meaning that there is a map $q:Y \to Z$ such that $q \circ i_Z = \id$.  Then the equivalence between $\Perf(Z)$ and $\DMF(S,F)$ can be presented in a different way.  Indeed, since there is a projection $q:Y \to Z$, there is a projection $\wt{q}:S \to Z$ obtained by composing $p$ and $q$.  Clearly $\wt{q}$ is $\Gamma$ invariant.  Consider the DG functor
$$ \Phi:\DBr(Z,0) \to \DBr(S,F) $$
defined by
$$ \Phi(\cP) = \wt{q}^*\cP \tensor \cK. $$
There is a natural isomorphism of functors $\Phi \cong i_* p^*$ since
$$ \wt{q}^*\cP \tensor \cK \to \wt{q}^*\cP \tensor i_*p^*\O_Z \to i_* p^*\cP $$
is an isomorphism in $\DQcoh(S,F)$.  Therefore $\Phi$ is an equivalence.

This situation occurs in one degenerate case.  Let $Z$ be a variety and let $\ol{\cV}$ be a vector bundle on $Z$.  Let $Y$ be the total space of $\ol{\cV}$ and set $q:Y \to Z$ be the vector bundle projection.  Now, put $\cV = q^*\cV$.  As we know, $\cV$ has a tautological section $s$ that vanishes to first order along $Z \subset Y$.  The structure of the LG pair $(S,F)$ corresponding to this $(\cV,s)$ is very special.  First, $S$ can be understood as the total space of $\ol{\cV} \oplus \ol{\cV}^\vee$.  Working over a small enough open affine set $U \cong \spec(R) \subset Z$, we can choose dual trivializations of $\ol{\cV} \oplus \ol{\cV}^\vee$ so that over $U$, $S$ looks like $\spec(R[x_1,\dotsc,x_r,y_1,\dotsc,y_r])$ and $F = \sum_i{ x_iy_i }$, where $r$ is the rank of $\ol{\cV}$.  If $x_i$ and $y_i$ correspond to $\ol{\cV}$ and $\ol{\cV}^\vee$ respectively then the $x_i$ and $y_i$ have $\Gamma$ weight 0 and 2 respectively while $R$ has $\Gamma$-weight 0.

This type of function $F$ was considered by Kn\"{o}rrer in \cite{Kn} and we can view Theorem \ref{main-equivalence} as a generalization of Kn\"{o}rrer's Theorem 3.1.
\end{rmk}

\begin{rmk}
It follows from Lemma \ref{Koszul} that $i_*p^*$ gives an equivalence between $\Perf(Z) = \langle \O_Z(i) \rangle_{i \in \Z}$ and $\langle \cK(i) \rangle_{i \in \Z}$.  Thus, if $Z$ is nonsingular, $\DMF(S,F) = \langle \cK(i) \rangle_{i \in Z}$.
\end{rmk}

%% file: GO_def.tex
\def\nzs{N_{Z/S}}

In this section we discuss a framework for handling a certain type of deformation of LG pairs.  As a special case, we show that the general LG pair from the previous section can be deformed to the degenerate LG pair of Remark \ref{kp-remark}.

Suppose that $(S,F)$ is an LG pair.  Let $Z$ be the reduced subscheme defined by $\tau(F)$.  \emph{We assume that $S$ is nonsingular and quasiprojective.}  Suppose that $Z$ is $\Gamma$ invariant and that $\cI_Z$ ($=\sqrt{\tau(F)}$) is the ideal sheaf defining $Z$.  Let $\nzs$ be the normal cone, the spectrum of the sheaf of algebras $\bigoplus_n \cI_Z^n/\cI_Z^{n+1}$.  Since $\cI_Z$ is $\Cx$ equivariant, there is a natural $\Gamma$ action on $\nzs$ inherited from $S$.  Let $d$ be the largest natural number such that $F \in \cI_Z^d$, so $F$ defines a nonzero section of $\cI_Z^d/\cI_Z^{d+1} \subset \O_{\nzs}$.  By abuse of notation we denote this regular function on $\nzs$ by $F$.  Under the inherited $\Gamma$ action, $F$ has degree 2.  Hence we obtain a new LG pair $(\nzs,F)$.

Consider the sheaf of algebras on $S$ given by
\[ \O_S[t,t^{-1}\cI_Z] = \dotsm \oplus t^{-2}\cI_Z^2 \oplus t^{-1} \cI_Z \oplus \O_S \oplus t\O_S \oplus \dotsm \]
and let $\St$ be the spectrum of this sheaf of algebras.  Note that $\St$ admits a map $\pi:\St \to \A^1$.  Write $\St_\lambda = \pi^{-1}(\lambda)$ for any point $\lambda \in \A^1$.  This map has the property that $\St_\lambda \cong S$ for any $\lambda \neq 0$ while $\St_0= \nzs$.  For this reason $\St$ is called the \emph{deformation to the normal cone}.

Since $\cI_Z$ is $\Gamma$-equivariant, the sheaf of algebras $\O_S[t,t^{-1}\cI_Z]$ is $\Gamma$-equivariant.  Hence $\St$ carries a $\Gamma$ action.  Observe that $t$ is fixed by $\Gamma$.  Each fiber $\St_\lambda$ is $\Gamma$ invariant and the induced $\Gamma$ actions on $\St_1$ and $\St_0$ agree with the actions we have already considered. Observe that $t^{-d}F$ is regular function on $\St$ having degree 2 for the $\Gamma$ action on $\St$.  So we obtain an LG pair $(\St,t^{-d}F)$.  The function $t^{-d}F$ has the property that its restrictions to $\St_1$ and to $\St_0$ are the functions we are calling $F$, by abuse of notation.  Hence the inclusions $(S,F) \to (\St,t^{-d}F)$ and $(\nzs,F) \to (\St,t^{-d}F)$ are morphisms of LG pairs.  

Assume now that $d=2$.  We write $T := \Cx$ and we will consider an action of $T$ on the LG pair $(\St,\Ft)$ where $\Ft := t^{-2}F$.  The notation is to avoid confusion between the two $\Cx$ actions that exist in this setting.  First note that there is a $\Cx$ action on $\St$ lifting the $\Cx$ action on $\A^1$.  This corresponds to the graded structure on the sheaf of algebras $\O_S[t,t^{-1}\cI_Z]$ where $\deg(t) = 1$.  To obtain the $T$ action we combine this action with the extant action.  We have constructed a $\Gamma \times \Cx$ action on $\St$ and we let $T$ act via the diagonal homomorphism $T \to \Gamma \times \Cx = \Cx \times \Cx$, $\lambda \to (\lambda,\lambda)$.  Under this action $t$ has weight one and $\Ft$ is $T$ invariant.

The variety $\St$ has commuting actions of $T$ and $\Gamma$.  Recall that a graded D-brane $\cE$ on $([\St/T],\Ft)$ is a graded D-brane on $(\St,\Ft)$ with a $T$-equivariant structure such that $d_\cE$ is $T$-equivariant.  We now consider the category
\[ \DBr([\St/T],\Ft) \]
whose objects are $T$-equivariant graded D-branes and where the complex of morphisms between two $T$-equivariant graded D-branes $\cE,\cF$ is
\[
 \Hom_{\DBr}(\cE,\cF)^T,
\]
the subcomplex of $T$-invariants.  Associated to the morphisms of LG pairs
$$ (S,F),(\nzs,F) \into (\St,\Ft) \to ([\St/T],\Ft) $$
we have pullback functors
\[
 \DBr([\St/T],\Ft) \to \DBr(S,F), \DBr(\nzs,F).
\]

Suppose that $(S,F)$ is obtained as in the previous section.  This means that there is a smooth quasiprojective variety $Y$, and a vector bundle $\cV$ over $Y$ with a regular section $s$.  Then $S$ is the total space of $\cV^\vee$ and $F$ is the function corresponding to the section $s$.  The $\Cx$ action on $S$ is derived from the natural action of $\Cx$ on $\cV^\vee$ by having $\lambda$ in the new action act by $\lambda^2$ in the old action.  Let $Z \subset Y$ be the zero locus of $s$ and view $Z$ as embedded in $S$ along the zero section.

Consider the LG pair $(\nzs,F)$.  In this situation
\[
 \nzs = N_{Z/Y}\oplus \cV^\vee|_Z \cong \cV|_Z \oplus \cV^\vee|_Z
\]
and the induced grading comes from doubling the natural action of $\Cx$ by scaling the $\cV^\vee$ summand and fixing the $\cV$ summand.  Moreover, the function $F$ on $S$ comes from contracting a point of $\cV^\vee$ with the section $s$.  Since $s$ vanishes along $Z$, the induced function on the normal bundle comes from contracting the $\cV|_Z$ summand with the $\cV^\vee|_Z$ summand.  Hence the LG pair $(\nzs,F)$ has the form considered in Remark \ref{kp-remark}.

Recall that there are canonical graded D-branes $\cK_N$ and $\cK_S$ on $(\nzs,F)$ and $(S,F)$ respectively.  In fact, we can interpolate between these with a graded D-brane on $([\St/T],\Ft)$.  Indeed, we form
\[ \cK_{\St} = (\wedge^\bt \wt{\cV}, t^{-1} s\wedge + t^{-1} \alpha \vee) \]
where $\wt{\cV}$ here denotes the pullback of $\cV$ to $\St$ under the natural map $\St \to S \to Y$.  This is equivariant since $s$ and $\alpha$ both have degree 1 for the LG $\Cx$ action.  We have $\cK_{\St}|_{\St_1} = \cK_S$ and $\cK_{\St}|_{\St_0} = \cK_N$.  Therefore $\cK_{\St}$ deforms $\cK_S$ into $\S$.

Let $\O(1)$ denote an arbitrary very ample line bundle on $Y$ and let $\O_{\St}(1)$ be its pull back under the map $\St \to Y$.  Since $\St \to Y$ is both $T$ and $\Gamma$-equivariant, $\O_{\St}(1)$ is both $T$ and $\Gamma$-equivariant.  Therefore, if $\cE$ is a $T$-equivariant graded D-brane on $(\St,\Ft)$ then so is $\cE(n) := \cE \tensor \O_{\St}(1)^{\tensor n}$ for any $n \in \Z$.  Moreover, the restrictions $\O_{\St}(1)|_{\St_0}$ and $\O_{\St}(1)|_{\St_1}$ are ample on $\St_0$ and $\St_1$ respectively.

Observe that since $S$ is the total space of $\cV^\vee$, there is a closed immersion $\wt{i}:\cV^\vee|_Z\times \A^1 = \Sbr \times \A^1 \into \St$.  Note that $\Sbr \times \A^1$ is $T$ and $\Gamma$ invariant.  Let $\wt{p}:\Sbr \times \A^1 \to Z$ be the obvious projection, which is flat and affine.  The pullback of a complex of quasicoherent sheaves under $\wt{p}^*$ gives a curved graded quasicoherent sheaf on $([\Sbr \times \A^1 /T],0)$.  We can push such a sheaf forward to obtain an object of $\DQcoh([\St/T],\Ft)$.  Let $\cC_Z$ be the full subcategory of $\DBr([\St/T],\Ft)$ consisting of objects isomorphic in $\DQcoh([\St/T],\Ft)$ to those of the form $\wt{i}_*\wt{p}^*\cG$ where $\cG$ is a coherent complex on $Z$.

\begin{thm} \label{restrictions}
The restrictions
$$\cC_Z \to \DBr(S,F)$$
and
$$ \cC_Z \to \DBr(N_{Z/S},F) $$
are quasi-equivalences.
\end{thm}
\begin{proof}
Given a coherent complex $\cG$ on $Z$, there is a finite complex $\cE^\bt$ of graded D-branes that resolves $\wt{i}_*\wt{p}^*\cG$, as in \ref{generation}.  Consider the restriction of $\cE^\bt \to \wt{i}_*\wt{p}^*\cG$ to $S = \St_1$.  Note that restriction of $T$-equivariant sheaves on $\St$ to $S$ is the same as restriction to $\St\setminus \St_0 \cong S \times \Cx$, followed by taking $T$ invariants.  Hence it is exact and we see that $\cE^\bt|_S$ resolves $\wt{i}_*\wt{p}^*\cG|_S = (i_1)_*(p_1)^*\cG$.  Furthermore, a sheaf of the form $\wt{i}_*\wt{p}^*\cG$ has no $t$ torsion.  So the restriction of $\cE^\bt \to \wt{i}_*\wt{p}^*\cG$ to $\St_0 = \nzs$ remains exact as well.  Clearly $\wt{i}_*\wt{p}^*\cG|_{\nzs} = (i_0)_*(p_0)^*\cG$.  Hence the compositions of $\wt{i}_*\wt{p}^*$ with the restrictions to $S$ and $\nzs$ are just the two versions of the equivalence $i_*p^*$.  
On the other hand $\wt{i}_*\wt{p}^*$ is fully faithful by essentially the same argument as Theorem \ref{main-equivalence}, using the fact that $\cK_{\St}(i) \to \wt{i}_*\wt{p}^*\O_Z(i)$ is an isomorphism in $\DQcoh([\St/T],\Ft)$ and that $\sfH^\bt(Z\times \A^1,\O_{Z\times\A^1}(i))^T = \sfH^\bt(Z,\O_Z(i))$.
\end{proof}

\begin{rmk}
We can also prove directly that the restrictions are fully faithful on a certain category.
Indeed, I claim that the restriction functors
\[ \langle \cK_{\St}(i) \rangle_{i \in \Z}^\oplus \to \DQcoh(\nzs,F) \]
and
\[ \langle \cK_{\St}(i) \rangle_{i \in \Z}^\oplus \to \DQcoh(S,F) \]
are fully faithful.  Since the restriction functors preserve direct sums, to show that these functors are fully faithful it suffices to show that
\begin{align*}
 \Hom(\cK_{\St}(i),\cK_{\St}(j)) &\to \Hom(\cK_N(i),\cK_N(j)), \\
 \Hom(\cK_{\St}(i),\cK_{\St}(j)) &\to \Hom(\cK_S(i),\cK_S(j)).
\end{align*}
are isomorphisms.  Now observe that there is a $T$-equivariant quasi-isomorphism $\sEnd(\cK_{\St}) \to \O_{Z\times \A^1}$, since $\sEnd(\cK_{\St})$ is a Koszul complex.  Moreover, the restriction of this quasi-isomorphism gives the quasi-isomorphisms
\begin{align*}
\sEnd(\cK_S) &\to \O_Z, \\
\sEnd(\cK_N) &\to \O_Z.
\end{align*}
Using the main spectral sequence \eqref{main.ss} we see that the induced maps on the second page are
\[
 \xymatrix@C=0pt{
& \ar[dl] \sfH^\bt(\St, \O_{Z\times\A^1}(j-i))_\bt^T \ar[dr] & \\
\sfH^\bt(Z,\O_Z(j-i)) \cong \sfH^\bt(S, \O_Z(j-i))_\bt & & \sfH^\bt(\nzs, \O_Z(j-i))_\bt \cong \sfH^\bt(Z,\O_Z(j-i))
}
\]
Of course $\sfH^\bt(\St,\O_{Z\times\A^1}(j-i))_\bt = \sfH^\bt(Z,\O_Z(j-i))_\bt[t]$.  The $T$-invariants are just $\sfH^\bt(Z,\O_Z(j-i))$ and we see that these maps are isomorphisms at the second page of the relevant spectral sequences.  Hence the maps above are isomorphisms.

In particular, the restrictions are fully faithful on $\langle \cK_{\St}(i)\rangle^\oplus_{i \in \Z} \cap \DMF([\St/T],\Ft)$.  However, without comparing everything to $\Dcoh(Z)$ it is not clear that the category $\langle \cK_{\St}(i)\rangle^\oplus_{i \in \Z} \cap \DMF([\St/T],\Ft)$ is large enough to guarantee that every graded D-brane on $(S,F)$ or $(\nzs,F)$ is in the essential image of restriction.  Nonetheless, it follows from \ref{restrictions} that we have
$$\cC_Z = \langle \cK_{\St}(i)\rangle^\oplus_{i \in \Z} \cap \DMF([\St/T],\Ft).$$
\end{rmk}

%% file: GO_revisit.tex
Now, we will see how we may combine the results of the previous sections with Segal's theorem to derive Orlov's theorem.  Suppose $X \subset \P=\P^{N-1}$ is a smooth Calabi-Yau complete intersection. Let $W_1,\dotsc,W_r \in \C[x_1,\dotsc,x_N]$ be homogeneous equations for $X$ with $d_i = \deg(W_i)$.  The Calabi-Yau condition is $\sum_{i=1}^r{d_i} = N$.  There are several relevant LG pairs.  First, we can combine the $W_i$ into a section $s_W$ of the bundle $\oplus_{i=1}^r{\O(d_i)}$ on $\P$.  This section gives rise to a function $W$ on the total space $Y$ of the bundle $\oplus_{i=1}^r{\O(-d_i)}$.  This function is linear on each fiber of the projection $p:Y \to \P$.  Since $Y$ is the total space of a vector bundle it has an action of $\Gamma$.  However, as in section 3, we consider the new ``doubled'' action induced by the squaring endomorphism of $\Gamma$.  Let $\O_Y(a) = \pi^*\O(a)$ and note that $\oplus_{i=1}^r \O_Y(-d_i)$ has a tautological section $s$.  The function $W$ can be factored as
\[ \xymatrix{ \O_Y \ar[r]^(.3){s_W} & \bigoplus_{i=1}^r\O_Y(d_i) \ar[r]^(.6){\vee s} & \O_Y } \]
where $\vee s$ denotes contraction with $s$.

We now describe Segal's theorem.  To begin with we consider $V:=\spec(\C[x_1,\dotsc,x_N,p_1,\dotsc,p_r])$ with a $G :=\Cx$ action and a $\Gamma$ action.  (As in previous sections we attempt to reduce confusion by introducing the notation $G$ and $\Gamma$ to differentiate between $\Cx$ actions.)  Under the first action, $\deg(x_i) = 1$ and $\deg(p_i) = -d_i$.  Under the $\Gamma$ action, we have $\deg(x_i) = 0$ and $\deg(p_i) = 2$.  The function $F = \sum_{i=1}^r{p_iW_i}$ is fixed by $G$ and and $\Gamma$-weight 2.  There are two possible open sets of semistable points in the GIT sense in $V$ associated to the identity and inversion characters of $G$.  Write $V_+$ and $V_-$ for the points semistable with respect to the identity and inversion characters, respectively.  In general, $V_+$ and $V_-$ are the complements of the hyperplanes defined by the vanishing of the positive and negative weight variables respectively.  So in our case, $V_+$ is the complement of the plane $x_i=0$ and $V_-$ is the complement of the plane $p_j = 0$.  We see that $[V_+/G] \cong Y$. We will describe $[V_-/G]$ in more detail below.  Both semistable sets are $\Gamma$ invariant and hence we obtain three LG pairs $([V/G],F),(Y,W),([V_-/G],W)$ fitting into a diagram
\[
 \xymatrix{ \DBr([V_-/G],W) & \ar[l]_(.45){j^*} \DBr([V/G],F) \ar[r]^(.55){j^*} & \DBr(Y,W) }
\]
Let $\mc{G}_t$ be the full DG subcategory of $\DBr([V/G],F)$ whose objects are graded D-branes $\cE$ whose underlying $G$-equivariant vector bundle is a direct sum of character line bundles in the set $\O_V(t),\dotsc,\O_V(t+N-1)$.  We can now formulate Segal's theorem.
\begin{thm*}[3.3, \cite{Se}]
The functors
\[
 \xymatrix{ \DBr([V_-/G],W) & \ar[l]_(.2){j^*} \mc{G}_t \ar[r]^(.35){j^*} & \DBr(Y,W) }
\]
are quasi-equivalences.
\end{thm*}

\begin{rmk}
This is a special case of the theorem that Segal proves.  He considers a linear action of $\Cx$ on a vector space $V$.  He assumes the action is Calabi-Yau, which means that the action map $\Cx \to \GL(V)$ factors through a $\SL(V) \subset \GL(V)$.  Then for a general $\Gamma$ action and potential $W$ on $V$, he constructs a family of equivalences as above between $\DBr([V_+/G],F)$ and $\DBr([V_-/G],F)$.
\end{rmk}

In conclusion, we can summarize the geometric picture in the following diagram, where the solid arrows are DG functors which are quasi-equivalences when labeled by $\he$.  The dashed lines indicate the ``phenomena responsible for the various equivalences and comparisons and the dotted arrow on the left represents the fully faithful functor between the homotopy categories of $\Perf(X)$ and $\DBr([V_-/G],W)$ that one obtains by going around the diagram counter clockwise.
\[
\xymatrix@C=5pt{ & \ar[dl]_{j_-^*}^\he \mc{G}_t \ar[dr]^{j_+^*}_\he & \\
\DBr([V_-/G],W) \ar@{--}[rr]^{\txt{``Segal inversion''}} & & \DBr(Y,W) \ar@{--}@/_1pc/[ddl]_{\txt{Deformation}} \\
\Perf(X) \ar[dr]^{\txt{Kn\"{o}rrer \\ periodicity}}_\he \ar@{..>}[u]^{\txt{Orlov type theorem}} & & \cC_Z \subset \DBr(\wt{Y},t^{-2}W)^T \ar[u]^{j^*}_{\he} \ar[dl]^{j^*}_\he \\
& \DBr(N_{X/Y}, W) }
\]
The quasi-equivalences induce triangulated equivalences in the homotopy categories.  There is another picture at the level of homotopy categories.
$$
\xymatrix{
\Dcoh(X) \ar@{..>}[r]\ar[d]^{i_*p^*} & [\DBr([V_-/G],W)] \\
[\DBr(Y,W)] & [\cG_t] \ar[u]^{j^*} \ar[l]^{j^*}
}
$$

If $r = 1$, $[V_-/G]$ has a simple description and $[\DBr([V_-/G],W)]$ is naturally equivalent to the category of graded matrix factorizations.  In this case $V=\spec(\C[x_1,\dotsc,x_N,p])$ and $V_- = \spec(\C[x_1,\dotsc,x_N,p,p^{-1}])$.  The ring $\C[x_1,\dotsc,x_n,p]$ has two gradings and the degrees are $\deg(x_i)=(1,0)$ and $\deg(p)=(-N,2)$.  Let $R=\C[x_1,\dotsc,x_N]$.  The finitely generated bigraded projective modules over $R[p^{-1}]$ are direct sums of the modules $R[p^{-1}](a,b)$, the free $R[p^{-1}]$ module generated by an element of degree $(a,b)$.  Note that $R[p^{-1}]$ is also generated by $p$ as a module and therefore $R[p^{-1}] \cong R[p^{-1}](-N,2)$ as bigraded modules.  The only units of $R[p^{-1}]$ not in degree zero are the powers of $p$ and hence there are no isomorphisms between the modules in the collection $R[p^{-1}](a,b)$ that do not come from $R[p^{-1}] \cong R[p^{-1}](-N,2)$ by shifts and compositions.  We see that the bigraded modules $R[p^{-1}](a,0), R[p^{-1}](a,1)$ are all distinct and that every bigraded projective module is isomorphic to a direct sum of these.

An object of $\DBr([V_-/G],W)$ is a bigraded projective $R[p^{-1}]$ module $\cE = \bigoplus_j R[p^{-1}](a_j,b_j)$ and an endomorphism $d$ of degree $(0,1)$ satisfying $d^2= pf$, where $f$ is the defining equation of our hypersurface.  Clearly
\begin{align*}
\Hom_{R[p^{-1}]}(R[p^{-1}](a,0),R[p^{-1}](a',0))_{(0,1)} &= \Hom_{R[p^{-1}]}(R[p^{-1}](a,1),R[p^{-1}](a',1))_{(0,1)} = 0 \\
\Hom_{R[p^{-1}]}(R[p^{-1}](a,1),R[p^{-1}](a',0))_{(0,1)} &= (R[p^{-1}])_{(a'-a,0)} = R_{a'-a} \\
\Hom_{R[p^{-1}]}(R[p^{-1}](a,0),R[p^{-1}](a',1))_{(0,1)} &= p(R[p^{-1}])_{(a'-a,0)} = p R_{a'-a}
\end{align*}
This means that the data to specify a graded D-brane on $([V_-/G],W)$ is the same as the data required to specify a graded matrix factorization of $f$ over $R=\C[x_1,\dotsc,x_N]$.  Given two graded D-branes $\cE,\cF$ on $([V_-/G],W)$ we write $E,F$, respectively, for the corresponding $R[p^{-1}]$ modules.  Since
\[ \Hom_{\DBr}(\cE,\cF) = \Hom_{R[p^{-1}]}(E,F)_{(0,*)} \]
we see that $[\Hom_{\DBr}(\cE,\cF)]$ is given by the space of graded chain maps $E \to F$ modulo nullhomotopic chain maps.  Hence $[\DBr([V_-/G],W)]$ is equivalent to the category of graded matrix factorizations and clearly the equivalence is compatible with the triangulated structure.

So we obtain as a corollary:
\begin{cor}
There is a family of equivalences $\Dcoh(X) \he \MF(W)$ when $X$ is a nonsingular Calabi-Yau hypersurface of a projective space, with defining equation $W$.
\end{cor}

%% file: GO_loc.tex
In this section we will formulate and prove a precise version of the statement that for an LG pair $(S,F)$ the category $\DBr(S,F)$ only depends on a formal neighborhood of the singular locus of the zero locus of $F$, when $S$ is quasi-projective.  To make this precise we need a notion of graded D-brane that makes sense on a formal neighborhood of the zero locus of $F$.

Let $(S,F)$ be an LG pair with $S$ nonsingular and quasi-projective with an equivariant ample line bundle $\cL$.  Recall that the \emph{Jacobi ideal (sheaf)} $J(F)$ of $F$ is defined to be the image of the map $\mc{T}_S \to \O_S$ given by contraction with $dF$, where $\mc{T}_S$ is the tangent sheaf.  The \emph{Tyurina ideal (sheaf)} is defined to be $\tau(F) := J(F) + F\cdot \O_S$.  The Tyurina ideal sheaf defines the scheme theoretical singular locus of the zero locus of $F$.  Let $Z$ be the reduced subscheme associated to $\tau(F)$.  Observe that $\tau(F)$ is $\Cx$-equivariant and hence $Z$ is invariant.

We consider the subschemes $Z^{(n)}$ defined by $\tau(F)^n$.  All of these schemes have an action $\Cx$ so that the closed immersions $Z^{(n)} \to S$ are equivariant.  Let $\Zh$ be the formal completion of $S$ along $Z$, where we choose $\tau(F)$ for the ideal of definition.
\begin{dfn}
An equivariant structure on a coherent sheaf $\cE$ on $\Zh$ is, for each $n$, an equivariant structure on $\cE|_{Z^{(n)}}$ such that the equivariant structure on $\cE|_{Z^{(n)}}$ is obtained by restriction from the equivariant structure on $\cE|_{Z^{(n+1)}}$.
\end{dfn}
View $F$ as a function on $\Zh$.  Now we can formulate the correct notion of a graded D-brane on $(\Zh,F)$.
\begin{dfn}
A \emph{graded D-brane on $\Zh$ controlled by $\cL$} is an equivariant vector bundle $\cE$ on $\Zh$ with an endomorphism $d_\cE$ of degree one such that $d_\cE^2 = F \cdot \id_\cE$ and such that for some $m \gg 0$ the natural map $\O_{\Zh} \tensor \Gamma(\Zh,\cE\tensor \cL^{\tensor m}) \to \cE \tensor \cL^{\tensor m}$ is surjective.
\end{dfn}
It remains to construct a DG category.  Let $U' \subset S$ be an invariant open affine and set $U = U' \cap Z$.  Then we define the graded ring
\[
 \O_{\Zh}^{gr}(U) = \bigoplus_{k\in\Z}{\invlim_n (\O_S(U')/\tau(F)^n)_k}
\]
If $\cE$ is an equivariant sheaf on $\Zh$ we can define the $\OZgr(U)$ module
\[
 \cE^{gr}(U) = \bigoplus_{k\in\Z}{\invlim_n (\cE(U')/\tau(F)^n\cE(U'))_k}.
\]
Suppose that $\cE$ and $\cF$ are two equivariant vector bundles on $\Zh$.  There is a natural graded $\OZgr(U)$-module structure on the space of continuous homomorphisms
\[ \Hom_{gr}(\cE,\cF)(U) := \Hom_{cont}(\cE^{gr}(U),\cF^{gr}(U)) \]
There is an alternate description
\[
\Hom_{gr}(\cE,\cF)(U) = \bigoplus_{k\in\Z}\invlim_n \left( \sHom( \cE, \cF )(U')/\tau(F)^n\sHom( \cE, \cF)(U') \right)_k
\]
The endomorphisms of $\cE$ and $\cF$ induce a differential on $\Hom_{gr}(\cE,\cF)(U)$ making it into a complex of $\C$ vector spaces and a DG $\OZgr(U)$-module.  Observe that the formation of $\Hom_{gr}(\cE,\cF)(U)$ is compatible with composition in the sense that there is canonical morphism
\begin{equation}\label{comp}
\Hom_{gr}(\cE_2,\cE_3)(U) \tensor_{\OZgr(U)} \Hom_{gr}(\cE_1,\cE_2)(U) \to \Hom_{gr}(\cE_1,\cE_3)(U)
\end{equation}
of DG $\OZgr(U)$-modules.

Fix a $\Cx$-invariant affine open cover $\{U_\alpha\}$ of $S$.
\begin{dfn}
The category $\DBr(\Zh,F,\cL)$ of graded D-branes on $(\Zh,F)$ controlled by $\cL$ is the DG category whose objects are graded D-branes on $(\Zh,F)$ controlled by $\cL$.  The complex of morphisms between $\cE$ and $\cF$ is the total complex of the bicomplex
\[
 \vC^\bt(\Zh,\{U_\alpha \cap Z\}, \Hom_{gr}(\cE,\cF))_\bt.
\]
Composition is induced by \eqref{comp}
\end{dfn}

Write $j:\Zh \to S$ for the natural morphism of locally ringed spaces.  If $\cE$ is a graded D-brane on $(S,F)$ then $j^*\cE$ is a graded D-brane controlled by $\cL$.  Moreover, if $\cE,\cF$ are two graded D-branes on $(S,F)$ and $U$ is an invariant open affine, there is a map
\[
 j^*\sHom(\cE,\cF)(U) \to \Hom_{gr}(j^*\cE,j^*\cF)(U \cap Z)
\]
of graded $\O_S(U)$ modules that intertwines the natural differentials.  This is compatible with compositions and defines a functor $j^*:\DBr(S,F) \to \DBr(\Zh,F,\cL)$.

\begin{thm}
The completion functor $j^*:\DBr(S,F) \to \DBr(\Zh,F,\cL)$ is a quasi-equivalence.
\end{thm}
\begin{proof}
We must verify that $j^*$ is quasi-fully faithful and quasi-essentially surjective.  To prove that $j^*$ is quasi-fully faithful we will check that
\[
 j^*:\sHom(\cE,\cF)(U_\alpha) \to \Hom_{gr}(j^*\cE,j^*\cF)(U_\alpha \cap Z)
\]
is a quasi-isomorphism.  Since $j^*$ is compatible with the filtrations by \v{C}ech degrees it induces a map of spectral sequences.  When the above map is a quasi-isomorphism for each $\alpha$ the map of spectral sequences becomes an isomorphism at the first page and hence $j^*$ is a quasi-isomorphism.  An exact sequence of graded modules is exact in each homogeneous degree.  Moreover, the inverse systems appearing in the definition of the graded completion satisfy the Mittag-Leffler condition.  Hence, graded completion is exact.  It follows that
\[
 \bigoplus_{k \in \Z} \invlim_n {\cH(\sHom(\cE,\cF)(U'))/\tau^n(F)\cH(\sHom(\cE,\cF))(U')_k} \cong \sfH^*\Hom_{gr}(j^*\cE,j^*\cF)(U)
\]
However, $\tau^n(F)\cH(\sHom(\cE,\cF))(U') = 0$ for all $n \geq 1$ and therefore
\[ \cH(\Hom(\cE,\cF))(U') \to \sfH^*\Hom_{gr}(j^*\cE,j^*\cF)(U) \]
is an isomorphism.

Now we must show that $j^*$ is quasi-essentially surjective.  We will deduce this from Theorem 3.10 of \cite{Or}, which we view as a local statement.  The theorem says that if $B$ is a graded ring of finite homological dimension and $W$ is a homogeneous element then
\[ \coker:\sfH^0\DBr(B,W) \to D_{Sg}^{gr}(B/WB) \]
is a triangulated equivalence.  We can decompose a graded D-brane $\cE$ into its odd and even parts $\cE^0$ and $\cE^1$ and we denote the restriction of $d_\cE$ to $\cE^0$ by $d_\cE^+:\cE^0 \to \cE^1$. The functor in the theorem is given by the assignment $E \mapsto \coker(d_{\cE}^+)$, which Orlov proves descends to a functor $\sfH^0\DBr(B,W) \to D_{Sg}^{gr}(B/WB)$.

Consider a graded D-brane $\cE$ on $(\Zh,F)$ controlled by $\cL$.  Write $V(F)$ for the subscheme defined by $F$ and $(m)$ for tensoring with $\cL^{\tensor m}$.  Let $\wh{\alpha} = \coker(d^+_\cE)$ and let $\alpha$ be a coherent equivariant sheaf on $V(F)$ such that $j^*\alpha = \wh{\alpha}$.  Suppose that
\[
 0 \to Q_k \to Q_{k-1} \to \dotsm \to Q_0 \to \alpha \to 0
\]
is an exact sequence of coherent sheaves on $V(F)$ such that $Q_i$ is locally free and equivariant for $i < k$.  Take $m \gg 0$ such that $Q_i(m)$ and $\cE(m)$ are globally generated.  Choose an equivariant map
\[
 \O_{\Zh}(-m)\tensor_\C \Gamma(S, Q_1(m)) \to \cE_+
\]
such that each square of
\[
 \xymatrix@C=.5cm{
\O_{\Zh}(-m)\tensor_\C \Gamma(S,Q_k(m)) \ar[r]\ar@{->>}[d] & \O_{\Zh}(-m)\tensor_\C \Gamma(S,Q_k(m)) \ar[r]\ar@{->>}[d] & \dotsm \ar[r] & \O_{\Zh}(-m)\tensor_\C \Gamma(S,Q_k(m)) \ar[r]\ar@{->>}[d] & \cE_+ \ar@{->>}[d] \\
j^*Q_k \ar[r] & j^*Q_{k-1} \ar[r] & \dotsm \ar[r] & j^*Q_1 \ar[r] & \alpha
}
\]
commutes and
\[
 \O_{\Zh}(-m)\tensor_\C \Gamma(S, Q_2(m)) \to \O_{\Zh}(-m)\tensor_\C \Gamma(S, Q_1(m)) \to \cE_+
\]
is zero.  Define $P_i = \ker( \O_S(-m)\tensor_\C \Gamma(S,Q_i(m)) \to Q_i)$.  If $i < k$ then $P_i$ is an equivariant vector bundle.  Moreover, if $k \geq \dim(S)-1$ then $P_k$ is also locally free.  Note that $P_k$ fits into an exact sequence
\[ 0 \to P_k \to \O_S(-m)\tensor_\C \Gamma(S,Q_k(m)) \to Q_{k-1} \to \dotsc \to Q_0 \to \alpha \to 0. \]
Now, the two term resolutions
\[ 0 \to P_i \to \O_S(-m)\tensor_\C \Gamma(S,Q_i(m)) \to Q_i \to 0 \]
imply that for any quasi-coherent sheaf $\beta$ on $S$, $\sExt^m(Q_i,\beta) = 0$ if $m > 1$.  Therefore
\[
 \sExt^1(P_k,\beta) \cong \sExt^{k+2}(\alpha,\beta) = 0
\]
and $P_k$ is locally free.

Suppose that $\phi: P \to Q(1)$ is an injective equivariant map of equivariant vector bundles on $S$ with the property that $F Q \subset \phi(P)$.  Over an invariant affine open, $P$ and $Q$ are graded projective modules.  We can define a map in the opposite direction $\psi:Q \to P(1)$ by $\psi(q) = \phi^{-1}(Wq)$.  This new map is equivariant and by construction $\phi \circ \psi = W \cdot \id_Q$ and $\psi \circ \phi = W \cdot \id_P$.  It is the unique such map and hence all of the local maps patch together to give a map $\psi:Q \to P(1)$.  So for each $i$ there is a unique equivariant arrow $\O_S(-m)\tensor_\C\Gamma(S,Q_i(m)) \to P_i$ which gives $\O_S(-m)\tensor_\C\Gamma(S,Q_i(m)) \oplus P_i[1]$ the structure of a graded D-brane which we denote $M_i$.  Observe that for $1 \leq i < k$, $M_i$ is locally contractible  since the cokernel of $ P_i \to \O_S(-m)\tensor_\C \Gamma(S,Q_i(m))$ is locally free on $V(F)$.

Let $C_{k-1} = \cn(M_k \to M_{k-1})$.  Since $M_{i+2} \to M_{i+1} \to M_i$ is the zero map we can inductively define $C_i = \cn(C_{i+1} \to M_i)$.  Since $M_i$ is contractible if $i < k$, the natural map $C_i \to C_{i+1}[1]$ is an isomorphism in the homotopy category.  Now there is a map $j^*C_1 \to \cE$.  Consider the cone $C = \cn(j^*C_1 \to \cE)$.

Let $U' \subset S$ be an invariant affine open set and $U = U' \cap Z$.  Then since the functor
\[ \coker:\sfH^0\DBr(\OZgr(U),F) \to D_{Sg}^{gr}(\OZgr(U)/(F)) \]
is triangulated, it follows from the construction of $C$ as an iterated cone that $\coker(C)$ is isomorphic to the acyclic complex
\[
 0 \to Q_k \to Q_1 \to \dotsm \to Q_1 \to \alpha \to 0
\]
which is itself isomorphic to zero.  Since $\coker$ is fully faithful, this implies that $C(U)$ is itself contractible.  Now since $C$ is locally contractible it is zero in the homotopy category $\sfH^0\DBr(\Zh,F)$.  This means that $j^*C_1$ is isomorphic to $\cE$ in the homotopy category.  Hence $j^*$ is quasi-essentially surjective.
\end{proof}

\begin{rmk}
Orlov \cite{Or4} has obtained a similar theorem in the case of categories of singularities.
\end{rmk}